\RequirePackage{fix-cm}
\documentclass{article}
\usepackage{amsmath,amsthm}
\usepackage{graphicx}

\theoremstyle{plain}
\newtheorem{theorem}{Theorem}

\newtheorem{corollary}{Corollary}

\newtheorem{lemma}{Lemma}

\theoremstyle{definition}
\newtheorem{definition}{Definition}

\newtheorem{remark}{Remark}

\theoremstyle{plain}

\begin{document}
\title{Polynomial algorithm for the disjoint bilinear programming problem  with an acute-angled polytope for a disjoint subset}


\author{Dmitrii Lozovanu\footnote{Moldova State University, Institute of Mathematics and Computer Science, Academy str., 5, Chisinau, MD--2028, Moldova,
\texttt{lozovanu@math.md}}}



\date{February 12, 2025}

\maketitle

\begin{abstract} We consider the disjoint bilinear
programming problem in which one of the disjoint subsets has the
structure of an acute-angled polytope. An optimality criterion for
such a problem is formulated and proved, and based on this,  a
polynomial algorithm for its solving is proposed and grounded. We
show that the proposed algorithm can be efficiently used for
studying and solving the boolean linear programming problem and the
piecewise linear concave programming problem.

\end{abstract}

\section{Introduction and problem formulation}\label{ddd}
The main results of the article are concerned with studying and
solving  the disjoint programming problem in the case when one of
the disjoint subsets has the structure of an acute-angled polytope.
We formulate and prove an optimality criterion for this problem and,
based on this, propose a polynomial algorithm for its solving.
Furthermore, we show that these results can be efficiently used for
studying and solving the boolean programming problem and the
piecewise linear concave programming problem.

The formulation of the disjoint bilinear programming problem is the
following \cite{aajs,alt,ahjs,dak,l91,l07,whi}:

\medskip
\noindent\emph{Minimize}
\begin{equation}
z=x^TCy+g x+e y\label{eq1}
\end{equation}
\emph{subject to}
\begin{equation}
A x \le a,\ \ \ x \ge 0; \label{eq2}
\end{equation}
\begin{equation}
B y \le b,\ \ \ y \ge 0, \label{eq3}
\end{equation}

\medskip
\noindent where

 $\qquad\qquad \ \ C=(c_{ij})_{n\times m},\ \ \
A=(a_{ij})_{q\times n},\ \ \ B=(b_{ij})_{l\times m}, $
$$ a^T=(a_{10},
a_{20},\dots, a_{q0})\in R^q, \ \ b^T=(b_{10}, b_{20}, \dots,
b_{l0})\in R^l,$$
$$g=(g_1, g_2, \dots g_n)\in R^n,\ \  \ e=(e_1, e_2, \dots, e_m)\in
 R^m,$$
$$x^T=(x_1, x_2, \dots,x_n)\in R^n, \ \ y^T=(y_1, y_2, \dots, y_m)\in R^m.$$
Throughout this article we will assume that the solution sets $X$
and $Y$ of the corresponding systems (\ref{eq2}) and (\ref{eq3}) are
nonempty and bounded. Our  aim  is to formulate and prove an
optimality criterion for problem (\ref{eq1})-(\ref{eq3}) that can be
efficiently used for studying and solving the problem in the case
when the set of solutions $Y$ has the structure of an acute-angled
polytope. Moreover we show that for this problem there exist a
polynomial algorithm for its solving.

 \medskip
 Disjoint bilinear programming  model (\ref{eq1})-(\ref{eq3}) comprises   a  large class of integer and combinatorial optimization
problems, including the well-known NP-complete  problem of the
existence of a boolean solution for a given system of linear
inequalities
\begin{equation}
\left\{%
\arraycolsep 2pt
\begin{array}{ccl}
\sum\limits_{i=1}^n a_{ij}x_j \leq a_{i0}, \ \ \ \quad
i=1,2,\dots,q,\\[4mm]
\ \ \ \ \ \ \ \ x_j\geq 0, \ \ \ \  \ \  \ \ j=1,2,\dots,n.
\end{array}%
\right. \label{eq31}
\end{equation}
It is easy to show that this problem can be represented as the
following disjoint bilinear programming problem:\\
\medskip \noindent\emph{Minimize}
\begin{equation}\label{equat3}
 z=\sum\limits_{j=1}^n \left(x_jy_j + (1-x_j)(1-y_j)
  \right)
\end{equation}
\noindent \emph{subject to}
\begin{equation}\label{equat4}
\left\{%
\begin{array}{c}
\sum\limits_{j=1}^n a_{ij}x_j  \le   a_{i0}, \ \ \ \ \ i=1,2,\dots,q, \\[4mm]
\ \ \ 0 \le x_j \le 1 , \ \ \ \ \ \ \ \ j=1,2,\dots,n;
\end{array}%
\right. \end{equation}
\begin{equation}\label{equat5}
 \ \ \ \ \ 0 \leq y_j \leq 1 , \ \ \quad \ \ \  j=1,2,\dots,n.
\end{equation}
 The
relationship between this bilinear programming problem and the
 problem of determining a boolean solution of system (\ref{eq31}) is the
following: system (\ref{eq31}) \ has a boolean solution
$x^*=(x_1^*,x_2^*,\dots,x_n^*)$ \ if and only if
$x^*=(x_1^*,x_2^*,\dots, x_n^*)$ and $y^*=(y_1^*,y_2^*,\dots,
y_n^*)$ with $y_j^*=1-x_j^*,  \ j=1,2,\dots, n,$ represents an
optimal solution of the disjoint bilinear programming problem
(\ref{equat3})-(\ref{equat5}) where $z(x^*,y^*)=0$.

 To disjoint bilinear programming problem (\ref{eq1})-(\ref{eq3}) also the following classical boolean linear programming problem  can
 be reduced:\\
\medskip
\noindent\emph{Minimize}
\begin{equation}
z=\sum_{j=1}^n c_j x_j   \label{eq4}
\end{equation}
\noindent \emph{subject to}
\begin{equation}
\left\{%
\arraycolsep 2pt
\begin{array}{rcl}
\sum\limits_{i=1}^n a_{ij}x_j & \le &  a_{i0}, \  \ \ \ \ \quad
i=1,2,\dots,q,
\\[4mm]
\ \ \ \ \ x_j & \in & \{ 0,1\}, \ \  \ \ \ j=1,2,\dots,n.
\end{array}%
\right. \label{eq5}
\end{equation}
This problem has an optimal boolean solution \
$x^*=(x_1^*,x_2^*,\dots,x_n^*)$ \ if and only if system (\ref{eq5})
has a feasible boolean solution. If coefficients $a_{ij}$ and $c_j$
of problem \ (\ref{eq4}), (\ref{eq5})\ are integer, then
$z^*=z(x^*)\in[-nH,nH]$, where \ $H=\max \{|c_i|,i=1,2,\dots,n\}$.
Therefore the optimal solution of boolean linear programming problem
(\ref{eq4}),  (\ref{eq5}) can be found by solving a sequence
of disjoint bilinear programming problems\\
\emph{Minimize}
$$ z=\sum\limits_{j=1}^n (x_jy_j +(1-x_j)(1-y_j)$$
\emph{subject to}
$$\left\{%
\arraycolsep 2pt
\begin{array}{lcl}
\sum\limits_{j=1}^n c_jx_j\leq t_k,\\[2mm]
\sum\limits_{i=1}^n a_{ij}x_j \leq a_{i0}, \ \ \ \quad
i=1,2,\dots,q,\\[2mm]
\ \ \ 0\leq x_j\leq 1, \ \ \ \  \ \  \ \ j=1,2,\dots,n; \ \ \ \\[2mm]
\ \ \ 0\leq y_j\leq 1, \ \ \ \ \ \ \ \  j=1,2,\dots,n,
\end{array}%
\right. $$ with integer parameters $t_k$ on $[-nH,nH]$ by applying
the bisection method with a standard integer roundoff procedure for
$t_k$.

Boolean programming problem (\ref{eq4}),  (\ref{eq5})
can be also  formulated as  the following disjoint bilinear programming problem:\\
\medskip \noindent\emph{Minimize}
\begin{equation}
\overline{z}=\sum\limits_{j=1}^n c_j x_j + M \sum\limits_{j=1}^n
\left(x_jy_j + (1-x_j)(1-y_j)
  \right) \label{eq6}
\end{equation}
\noindent \emph{subject to}
\begin{equation}
\left\{%
\begin{array}{c}
\sum\limits_{j=1}^n a_{ij}x_j  \le   a_{i0}, \ \ \ \ \ i=1,2,\dots,q, \\[1mm]
\ \ \ 0 \le x_j \le 1 , \ \ \ \ \ \ \ \  j=1,2,\dots,n;
\end{array}%
\right. \label{eq7}
\end{equation}
\begin{equation}
\ \ \ \ \   0 \le y_j \le 1, \ \ \ \ \ \ \ \  j=1,2,\dots,n,
\label{eq8}
\end{equation}
where $M$ is a suitable large value. More precisely, if the
coefficients in (\ref{eq4}), (\ref{eq5}) are integer, then $M$
should satisfy the condition $M\geq n\cdot 2^{3L+1}$, where $L$ is
the length of the binary encoding of  the coefficients of boolean
problem  (\ref{eq31}),  (\ref{eq4}).
 The relationship between boolean linear programming problem  \ (\ref{eq4}), (\ref{eq5}) \ and  disjoint bilinear programming problem
\ (\ref{eq6})-(\ref{eq8}) \ is the following: boolean linear
programming problem  \ (\ref{eq4}),\ (\ref{eq5})\ has an optimal
solution \ $x^*=(x_1^*,x_2^*,\dots, x_n*)$ \ if and only if \
$x^*=(x_1^*,x_2^*,\dots,x_n^*)$ \ and \ $y^*=(y_1^*,y_2^*,\dots,
y_n^*),$ where  $y_i^*=1-x_i^*,   i=1,2,\dots, n,$ represents a
solution of disjoint bilinear programming problem
(\ref{eq6})-(\ref{eq8}) and
$\overline{z}(x^*,y^*)=\sum_{j=1}^nc_jx^*.$ So, the boolean
programming problem can be reduced in polynomial time to disjoint
programming problem (\ref{eq1})-(\ref{eq3}) where the matrix $B$ is
identity one.

\medskip
Another important  problem which can be reduced to  disjoint
bilinear programming problem (\ref{eq1})-(\ref{eq3}) is the
following
piecewise linear concave programming problem:\\
\emph{Minimize}
\begin{equation}
z=\sum_{j=1}^l \min \{c^{jk}x + c_0^{jk}, \ \ \ \ \ \
k=1,2,\dots,m_j\}
  \label{eq9}
\end{equation}
\noindent\emph{subject to (\ref{eq2})}, \ where $x\in R^n, \
c^{jk}\in R^n, \ c_{0}^{jk} \in R^1$.

 This problem arises as an
auxiliary one when solving a class of resource allocation problems
\cite{l91}. In \cite{l91} it is shown that this problem can be
replaced by the following disjoint bilinear programming problem:\\
\emph{Minimize}
$$z=\sum_{j=1}^l \sum_{k=1}^{m_j}(c^{jk}x+ c_0^{jk})y_{jk}$$
\emph{subject to}
$$Ax \le a,  \ \ x \ge 0;$$
$$\left\{\begin{array}{l}
\sum\limits_{k=1}^{m_j} y_{jk} =1 , \ \ \ \ \ \  j=1,2,\dots,l;\\[4mm]
 \quad\ \   y_{jk}  \ge  0 , \ \ \ \ \ \
k=1,2,\dots,m_j, \ \ \ j=1,2,\dots,l.
\end{array}\right.$$
In this problem we can eliminate
$y_{1,m_1},y_{2,m_2},\dots,y_{l,m_l}$, taking into account that
$$y_{jm_j}=1- \sum\limits_{k=1}^{m_j-1} y_{jk}, \quad  \quad   j=1,2,\dots l, $$
and we obtain the following disjoint bilinear programming problem:\\
\emph{Minimize}
$$z=\sum_{j=1}^l \sum_{k=1}^{m_j-1}(c^{jk}-c^{jm_j})xy_{jk}+\sum_{j=1}^lc^{jm_j}x+\quad\quad\quad\quad\quad\quad\quad\quad\quad\quad\quad\quad\quad\quad\quad$$
\begin{equation}
\quad\quad\quad\quad\quad\quad\quad
\quad\quad\quad\quad\quad\quad\quad\quad \ +\sum_{j=1}^l
\sum_{k=1}^{m_j-1}(c_0^{jk}-c_0^{jm_j})y_{jk}+\sum_{j=1}^lc_0^{jm_j}
  \label{eq10}
\end{equation}
\emph{subject to}
\begin{equation}
Ax \le a,  \ \ x \ge 0; \label{eq11}
\end{equation}
\begin{equation}
\left\{\begin{array}{l}
\sum\limits_{k=1}^{m_j -1} y_{jk} \leq 1 , \ \ \ \ \ \  j=1,2,\dots,l;\\[4mm]
 \quad \  \  \   y_{jk}  \ge  0 , \ \ \ \ \ \
k=1,2,\dots,l, \ \ \ k=1,2,\dots,m_j-1,
\end{array}\right.\label{eq12}
\end{equation}
i.e. we obtain a special case of  disjoint bilinear programming
problem (\ref{eq1})-(\ref{eq3}) where the corresponding matrix $B$
is step-diagonal.

\medskip
Disjoint bilinear programming problem (\ref{eq1})-(\ref{eq3}) has
been extensively studied in \   \cite{aajs,alt,ahjs,dak,l91,l07,whi}
\ and some general methods and algorithms have been developed. In
this article we shall use a new optimization criterion that takes
into account the structure of the disjoint subset $Y$.

\medskip
It can be observed that in the presented above examples of disjoint
programming problems the  matrix $B$  is either identity one or step
diagonal. This means that the set of solutions $Y$  has the
structure of an acute-angled polyhedron \cite{andre,coxet}.
Acute-angled polyhedra are polyhedra in which all dihedral angles
are acute or right. A detailed characterization of such polyhedra
has been made by Coxeter \cite{coxet} and Andreev
\cite{andre,andr1}. Moreover, in \cite{andre,andr1} it has been
proven that  in an acute-angled polyhedron the hyperplanes of
nonadjacent facets cannot intersect. Based on this property in the
present  article we show that for a disjoint bilinear programming
problem with the structure of an acute-angled polytope for a
disjoint subset  an optimality criterion can be formulated that can
be efficiently used for  studying and solving the problem. In fact
we show that the formulated optimality criterion is valid not only
for the case when $Y$  is an acute-angled polytope but it holds also
for a more general class of polyhedra for which the hyperplanes of
nonadjacent facets do not intersect. In this article such polyhedra
are called \emph{perfect polyhedra}; we call the corresponding
disjoint bilinear programming problems with such a disjoint subset
\emph{disjoint bilinear programming problems with a perfect disjoint
subset}. The aim of this paper is to show that for this class of
problems an optimality criterion can be formulated that takes into
account the mentioned structure of the disjoint subset $Y$ and that
can be efficiently used for solving the problem. Moreover, based on
presented optimality criteria for these problem we ground a
polynomial algorithms for their solving. The optimization criterion
for disjoint bilinear programming with an acute-angled polytope for
a disjoint subset has been announced in \cite{l07}. The full proof
of this optimization criterion and the results concerned with its
application for solving the disjoint bilinear programming problems
with a perfect disjoint subset we present in this article.

\medskip The article is
organized as follows. In Section \ref{sectD} we present the
formulation of the disjoint bilinear programming problem with a
prefect disjoint subset that generalizes the problem with a disjoint
subset having the structure of an acute-angled polytope. In Sections
\ref{minmax}-\ref{ttt} we present the basic properties of the
optimal solutions of the  disjoint bilinear programming problem
(\ref{eq1})-(\ref{eq3}) in general. An important result of Section
\ref{ttt} represents  Theorem \ref{tm01}  that gives an optimization
criterion for problem (\ref{eq1})-(\ref{eq3}) that allows to ground
a new approach for its solving. The general scheme of this approach
is presented in Section \ref{sybsect1}. Section \ref{rdund1}
provides some necessary auxiliary results related to redundant
inequalities for systems of linear inequalities. In Section
\ref{trusc2}  we formulate and prove the optimality criterion for
the disjoint bilinear programming problem with a perfect disjoint
subset;  the main results of this section represents Theorems
\ref{thm101}, \ref{cor102} on the basis of which new algorithms are
proposed. Sections \ref{polalg}-\ref{finita} provide results
concerned with elaboration of a polynomial algorithm for solving the
disjoint bilinear programming problem with a perfect disjoint
subset. In Section \ref{finita1} is shown how the proposed algorithm
can be used for solving the boolean programming problem and the
piecewise linear concave programming problem.
\section{Disjoint bilinear programming
with a perfect disjoint subset}\label{sectD}Let us consider the
following disjoint bilinear programming
problem:\\
\emph{Minimize} \begin{equation}\label{eq08}z=x^TCy+g x+e
y\end{equation}
 \noindent\emph{subject to}
\begin{equation}\label{eq28} A x \le a,\ \ \ x \ge 0;
\end{equation}
\begin{equation}
D y \le d.\quad\quad\quad \ \  \label{eq38}
\end{equation}
This problem differs from problem \ (\ref{eq1})-(\ref{eq3}) \ only
by system (\ref{eq38}) in which  $\!D\!=\!(d_{ij})_{p\times m}$ and
$d^T\!\!=\!(d_{10},d_{20}\dots,d_{p0})\in\!R^p;$ in this problem $A,
C, a, g, e$ are the same as in problem (\ref{eq1})-(\ref{eq3}). We
call this problem  \emph{a disjoint bilinear programming problem
with perfect disjoint subset $Y$}  if system (\ref{eq38}) has a full
rank equal to $m$, where $m<p$, and this system possesses the
following properties:

\medskip\noindent $\ \ \ \    $a) \emph{system (\ref{eq38}) does not contain redundant inequalities and its solution set}

$\ \   $ \emph{$Y$ is a  bounded  set with nonempty interior;}

\medskip\noindent
$ \ \  \ \    $b) \emph{the set of solutions \ $Y'$ \ of an
arbitrary subsystem \ $D'y\leq d'$ \ of rank \ $m$}

$\ \   $ \emph{with $m$ \ inequalities of system \ (\ref{eq38})\
represents a convex cone \ with \ the}

$\ \ \    $\emph{origin at an extreme point \ $y'$  of the set of
solutions \ $Y$ of system \ (\ref{eq38}),}

$\ \  $\emph{ i.e. $y'$ is the solution of the system of equations \
$D'y= d'$;}

\medskip\noindent
$\ \ \      $ c) \emph{at each extreme point \ $y'$ \ of \ the \ set
of \ solutions \ $Y$ \ of \ system \ (\ref{eq38})}

$\ \  $ \emph{exactly $m$ hyperplanes of the facets of the polytope
$Y$ intersect};

\medskip\noindent
The main properties of perfect polyhedra  as well of the disjoint
programming problem with a perfect disjoint subset are studied in
Section \ref{trusc2}. Note that if $Y=\{y \big|\ Dy\leq d\}$ \ is a
nonempty set that has the structure of a nondegenerate acute-angled
polytope, then it satisfies conditions a)-c) above.
\section{Disjoint bilinear programming and the  min-max  linear  problem
 with interdependent variables}\label{minmax} Disjoint
bilinear programming problem (\ref{eq1})-(\ref{eq3}) is tightly
connected with the
 following min-max linear programming problem with interdependent
 variables:

\medskip\noindent \emph{To find}
\begin{equation}
   z^*=\min_{y\in Y}\max_{u\in U(y)}(ey-a^Tu)
  \label{neweq4}
\end{equation}
\emph{and}
\begin{equation}\label{neweq50}
y^*\in Y=\{y| \ By\leq b, \ y\geq 0\}
\end{equation}
 \emph{for which}
\begin{equation}\label{neweq5}
z^*=\max_{u\in U(y^*)}(ey^*-a^Tu),
\end{equation}\
\emph{where}
\begin{equation}\label{neweq6}
U(y)=\{u\in R^q |\ -A^Tu\leq Cy+ g^T, \ \ u\geq 0 \}.
\end{equation}

\medskip
\noindent The relationship between the solutions of  problem
(\ref{eq1})-(\ref{eq3}) and min-max problem
(\ref{neweq4})-(\ref{neweq6}) can be  obtained on the basis of the
following theorems.
\begin{theorem}\label{extrthm1}
If $(x^*,y^*)$ is an optimal solution of  problem
(\ref{eq1})-(\ref{eq3})  and  $z^*$ is the optimal value of the
objective function in this problem, then $z^*$ and  $y^*\in Y$
 represent a solution of min-max problem
(\ref{neweq4})-(\ref{neweq6}) and vice versa: if $z^*$ and  $y^*\in
Y$   represent a solution of min-max problem
(\ref{neweq4})-(\ref{neweq6}), then $z^*$ is the optimal value of
the objective function of problem (\ref{eq1})-(\ref{eq3})  and $y^*$
corresponds to an optimal point in this problem.
\end{theorem}
\medskip\noindent
\begin{proof} The proof of the theorem is obtained from the following reduction procedure of  bilinear programming problem  (\ref{eq1})-(\ref{eq3}) to
min-max problem (\ref{neweq4})-(\ref{neweq6}). We represent disjoint
bilinear programming problem  (\ref{eq1})-(\ref{eq3}) as the problem
of determining
\begin{equation}\label{neweq01}
\psi_1(y^*)= \min_{y\in Y}\psi_1(y),
\end{equation}
 where
\begin{equation}
\left\{%
\begin{array}{ll}
   \psi_1(y)=\min\limits_{x\in X}(x^TCy+gx+ey),\\[3mm]
   \quad\quad\quad \quad X=\{x\in R^n|\ Ax\leq a, \ x\geq 0 \}.
\end{array}%
\right. \label{neeq1}
\end{equation}
If we replace  linear programming problem (\ref{neeq1}) with respect
to $x$ by the dual problem
   $$\psi_1(y)=\max_{u\in U(y)}(ey-a^Tu)$$
and after that we introduce this expression in (\ref{neweq01}), then
we obtain  min-max problem
(\ref{neweq4})-(\ref{neweq6}).
\end{proof}

Similarly we can prove the following result.
\begin{theorem}\label{extrthm2}
If $(x^*, y^*) \in X\times Y$ is an optimal solution of disjoint
bilinear programming problem (\ref{eq1})-(\ref{eq3}) and $z^*$ is
the minimal value of the objective function of this problem, then
$z^*$ and $x^*$ correspond to a solution of the following min-max
linear problem:\\
To find
\begin{equation}
    z^*=\min_{x\in X}\max_{v\in V(x)}(gx-b^Tv)
  \label{neweq1}
\end{equation}
and  \begin{equation}\label{neweq51} x^*\in X=\{x| \ Ax\leq a, \
x\geq 0\}
\end{equation} for which
\begin{equation}\label{neweq2}
z^*=\max_{v\in V(x^*)}(cx^*-b^Tv),
\end{equation}
where
\begin{equation}\label{neweq3}
V(x)=\{v\in R^l |\ -B^Tv \leq C^{T}x+e^T, \ \ v\geq 0\}.
\end{equation}
\end{theorem}
\begin{corollary}\label{minmax11}
If $(x^*, y^*) \in X\times Y$ is an optimal solution of the
following disjoint bilinear programming problem:\\
Minimize
$$z=x^TCy+g x+e y $$
subject to
$$A x \le a,\ \  x \ge 0, \ \ \  D y \le d,$$
 and $z^*$ is the minimal
value of the objective function of this problem, then $z^*$ and
$x^*$
correspond to a solution of the following min-max linear problem:\\
To find
$$z^*=\min_{x\in X}\max_{v\in V(x)}(gx-d^Tv)$$
 and
$$ x^*\in X=\{x| \ Ax\leq a, \ x\geq 0\}$$
for which $$z^*=\max_{v\in V(x^*)}(gx^*-d^Tv),$$ where $$V(x)=\{v\in
R^l |\ -D^Tv = C^{T}x+e^T, \ \ v\geq 0\}.$$
\end{corollary}
Problems (\ref{neweq4})-(\ref{neweq6}) and
(\ref{neweq1})-(\ref{neweq3}) can be regarded as a couple of dual
min-max linear problems with interdependent variables. It is easy to
see that these min-max problems always have solutions if $X$ and $Y$
are nonempty and bounded sets.
\section{Estimation of the optimal value of the objective
function for the disjoint bilinear programming problem
(\ref{eq1})-(\ref {eq3})}\label{estim1}
 Let $L$
be the length of the input data of disjoint bilinear programming
problem (\ref{eq1})-(\ref{eq3})\ with integer coefficients of
matrices \ $C, A, B$ \ and of vectors \ $a, b, c, e$
\cite{gar,karp}, i.e. $$L=L_1+L_2+ (qm+q+m)(1+\log(H+1)),$$ where
$$L_1=
\sum_{i=0}^q\sum_{j=1}^n\log(|a_{ij}|+1)+\log{n(q+1)},$$
$$L_2=\sum_{i=1}^l\sum_{j=0}^m\log(|b_{ij}|+1)+\log{l(m+1)},$$
$$\quad \ \  H=\max\{|c_{ij}|,|g_i|,|e_j|, \ i=\overline{1,n},\
j=\overline{1,m}\}.$$ Then the following lemma holds.
\begin{lemma}\label{lem1}If   disjoint bilinear programming problem  (\ref{eq1})-(\ref{eq3})  with
 integer coefficients has optimal solutions, then the optimal value $z^*$
of the objective function (\ref{eq1}) is a rational number that can
be expressed by an irreducible fraction $\displaystyle\frac{M}{N}$
with integer $M$ and $N \ (|N|\geq 1)$, where \ $|M|$ \ and \ $|N|$\
do not exceed \ $2^L$.
\end{lemma}
\begin{proof} If the optimal value of the objective function of problem (\ref{eq1})-(\ref{eq3})
exists, then this value is attained at an extreme point $(x',y')$ of
the polyhedral set $X\times Y$ determined by (\ref{eq2}),
(\ref{eq3}) where $x'\in X$ and $y'\in Y$ (see
\cite{aajs,alt,ahjs,dak}). Then according to Lemma 1 from \cite{k80}
(see also \cite{k82})  each component $x_i'$ of $x'$ is a rational
value and it can be expressed by a fraction of form \
$x'=\displaystyle\frac{M_i^1}{N_0^1}$ with integer $M_i^1$ and
$N_0^1$, where $M_i^1$ is a determinant of the extended matrix of
system (\ref{eq2}), $N_0^1$ is a nonzero determinant of matrix $A$
and $|M_i^1|, |N_0^1|\leq\displaystyle\frac{2^{L_1}}{n(q+1)}$;
 similarly each
component $y_j'$ of $y'$ is a rational value and it can be expressed
by a fraction of form \ $y'=\displaystyle\frac{M_j^2}{N_0^2}$
 with integer
$M_j^2$ and $N_0^2$, where $M_j$ is a determinant of the extended
matrix of system (\ref{eq3}), $N_0^2$ is a nonzero determinant of
matrix $B$ and
$|M_j^2|,|N_0^2|\leq\displaystyle\frac{2^{L_2}}{l(m+1)}$. Therefore
$$z^*=
\displaystyle\frac{1}{N_0^1N_0^2}\Big(\sum_{i=1}^q\sum_{j=1}^mc_{ij}M_i^1M_j^2+\sum_{i=1}^q
g_iM_i^1N_0^2+\sum_{j=1}^me_jM_j^2N_0^1\Big)$$ where
$$\big|N_0^1N_0^2|=\big|N_0^1\big|\big|N_0^2\big|\leq\displaystyle\frac{2^{L_1+L_2}}{n(q+1)l(m+1)}$$
and $$\Big|\sum_{i=1}^q\sum_{j=1}^mc_{ij}M_i^1M_j^2+\sum_{i=1}^q
g_iM_i^1N_0^2+\sum_{j=1}^me_jM_j^2N_0^1\Big|\leq$$
$$H\Big|\sum_{i=1}^q\sum_{j=1}^mM_i^1M_j^2+\sum_{i=1}^q
M_i^1N_0^2+\sum_{j=1}^mM_j^2N_0^1\Big|\leq$$
$$2^{\log(H+1)}\Big(2^{L_1+L_2}+2^{L_1+L_2}+2^{L_1+L_2}\Big)\leq 2^L.$$
So, the optimal value $z^*$ of the objective function (\ref{eq1}) is
a rational number  that can be represented  by a fraction
$\displaystyle\frac{M}{N}$ with integer $M$ and $N \ (|N|\geq 1)$,
where \ $|M|$ \ and \ $|N|$\ do not exceed \ $2^L$.
\end{proof}
\begin{remark}\label{mica}
A more suitable upper bound estimation for the optimal value of the
disjoint bilinear programming problem \ (\ref{eq1})-(\ref{eq3})\
when \ $X$ \ and \ $Y$ \ are nonempty and bounded sets may be the
value
\begin{equation}\label{nic1}M^1_u=\min_{x\in X}\max_{y\in Y}(x^TCy+g x+ey)=\max_{y\in
Y}\min_{x\in X}(x^TCy+g x+ey)\end{equation} or the value
\begin{equation}\label{nic2} M^2_u=\max_{x\in X}\min_{y\in Y}(x^TCy+g x+ey)=\min_{y\in Y}\max_{x\in X}(x^TCy+g
x+ey)\end{equation} that can be found in polynomial time because
problems (\ref{nic1}) and (\ref{nic2}) can be reduced to linear
programming problems similarly  as the problem of finding the value
of a matrix game. Therefore for the optimal value of the object
function of disjoint bilinear programming problem
(\ref{eq1})-(\ref{eq3}) we have the following estimation
$$-2^L\leq z^*\leq \min \{2^L,M^1_u, M^2_u\}.$$
\end{remark}
\section{An optimality criterion for disjoint bilinear
programming (\ref{eq1})-(\ref{eq3})}\label{ttt} Let us assume that
the optimal value of the objective function of disjoint bilinear
programming problem (\ref{eq1})-(\ref{eq3}) is bounded. Then problem
(\ref{eq1})-(\ref{eq3}) can be solved by varying the parameter \ $h
\in [-2^L,2^L]$ \ in the problem of determining the consistency
(compatibility) of the system
\begin{equation}
\left\{%
\begin{array}{ll}
A x \le a, \\
x^TCy+g x+e y \le h, \\
B y \le b, \\
x \ge 0, \ y \ge 0.
\end{array}%
\right. \label{eq13}
\end{equation}
In order to study the consistency problem for system (\ref{eq13}) we
will reduce it to a consistency problem for a system of linear
inequalities with a right-hand member depending on parameters using
the following results.
\begin{lemma}\label{thm0} Let solution sets $X$ and $Y$ of the corresponding systems
(\ref{eq2}) and (\ref{eq3}) be nonempty and bounded. Then system
(\ref{eq13}) for a given $h\in R^1$ has no solutions if and only if
the system of linear inequalities
\begin{equation}
\left\{%
\begin{array}{ll}
-{A}^T u \le Cy+g^T, \\
\quad     a^T u < e y-h,  \\
\quad \ \ \  u \ge 0, \\
\end{array}%
\right. \label{eq14}
\end{equation}
\noindent is consistent with respect to $u$ for every $y$ satisfying
(\ref{eq3}). \end{lemma}
\begin{proof} System (\ref{eq13})
has no solutions if and only if for every $y \in Y$ the system of
linear inequalities
\begin{equation}
\left\{%
\begin{array}{ll}
    A x \le a,\\
    x^T (Cy+g^T) \le h-e y,  \\
    x \ge 0, \\
\end{array}%
\right. \label{eq15}
\end{equation}
has no solutions with respect to $x$. If we apply the duality
principle (Theorem 2.14 from \cite{che}) for  system (\ref{eq15})
with respect to vector of variables $x$,  then we obtain that it is
inconsistent if and only if the system of linear inequalities
\begin{equation}
\left\{%
\begin{array}{ll}
{A}^T \lambda + (Cy+g^T) t \ge 0, \\
\  a^T \lambda + (h-e y) t < 0,  \\
\ \ \ \   \lambda \ge 0, \ t \ge 0, \\
\end{array}%
\right. \label{eq16}
\end{equation}
has solutions with respect to $\lambda$ and $t$ for every $y \in Y$.
Note that for an arbitrary solution $(\lambda^*,t^*)$ of system
(\ref{eq16}) the condition $t^*>0$ holds. Indeed, if $t^*=0$, then
it means that the system
\[
\left\{%
\begin{array}{ll}
    {A}^T \lambda \ge 0, \\
   \  a^T \lambda < 0,\\
   \ \ \ \ \lambda \ge 0, \\
\end{array}%
\right.
\]
has solutions, which, according to Theorem 2.14 from \cite{che},
involves the inconsistency of system (\ref{eq2}) that is contrary to
the initial assumption. Consequently, $t^*>0$. Since $t>0$ in
(\ref{eq16}) for every $y \in Y$, then, dividing each inequality of
this system by $t$ and denoting $u=(1/t)\lambda$, we obtain  system
(\ref{eq14}). So, system (\ref{eq13}) is inconsistent if and only if
system (\ref{eq14}) is consistent with respect to $u$ for every $y$
satisfying (\ref{eq3}).
\end{proof}
\begin{corollary}\label{corol01} Let solution sets $X$ and $Y$ of the corresponding systems
(\ref{eq2}) and (\ref{eq3}) be nonempty and bounded. Then for a
given $h$ system (\ref{eq13}) has solutions if and only if there
exists $y\in Y$ for which system (\ref{eq14})  is inconsistent with
respect to $u$. The minimal value $h^*$ of parameter $h$  for which
such a property  holds is equal to the optimal value of the
objective function in disjoint bilinear programming \ problem
(\ref{eq1})-(\ref{eq3}), i.e. in this \ case \ for $h=h^*$ there
exists \ $y=y^*\in Y$ \ such that system (\ref{eq14}) is
inconsistent with respect to $u$.
\end{corollary}
\begin{theorem}\label{tm01} Let solutions sets $X$ and $Y$ of the corresponding systems
(\ref{eq2}) and (\ref{eq3}) be nonempty and bounded. Then for a
given $h\in R^1$ the system
\begin{equation}
\left\{%
\begin{array}{ll}
A x \le a, \\
x^TCy+g x+e y < h, \\
B y \le b, \\
x \ge 0, \ y \ge 0.
\end{array}%
\right. \label{eq013}
\end{equation}
is inconsistent if and only if the  system of linear inequalities
\begin{equation}
\left\{%
\begin{array}{ll}
    -{A}^T u \le Cy+g^T, \\
 \quad   a^T u \le e y-h,  \\
 \quad \ \ \   u \ge 0, \\
\end{array}%
\right. \label{eq17}
\end{equation}
is consistent with respect to $u$ for every $y$ satisfying
(\ref{eq3}). The maximal value $h^*$ of the parameter $h$ for which
system (\ref{eq17}) is consistent with respect to $u$ for every
$y\in Y$ is equal to the minimal value of the objective function of
disjoint bilinear programming problem (\ref{eq1})-(\ref{eq3}).
Moreover,  for the considered systems the following properties hold:

\medskip
1) If the  system of linear inequalities
\begin{equation}
\left\{%
\begin{array}{ll}
-{A}^T u - Cy\ \le \ g^T, \\
\quad    a^Tu - \ e y \ < -h^*,  \\
\quad \ \ \  u \ge 0, \\
\end{array}%
\right. \label{eq140}
\end{equation}
is inconsistent with respect to $u$ and $y$, then the system
\begin{equation}
\left\{%
\begin{array}{ll}
    \ \ Ax  \le a, \\
    C^Tx=-e^T ,  \\
  \ \ gx = h^*,\\
 \quad \ x  \ge 0, \\
\end{array}%
\right. \label{eq171}
\end{equation}
is consistent and an arbitrary solution to it $x^*$ \ together with
an arbitrary \ $y\in Y$ \ determine a solution $(x^*,y)$ of disjoint
bilinear programming problem (\ref{eq1})-(\ref{eq3}).

\medskip
2) If system (\ref{eq140}) is consistent with respect to $u$ and
$y$, then there exists $y^*\in Y$ for which the system
\begin{equation}
\left\{%
\begin{array}{ll}
    -{A}^T u \le Cy^*+g^T, \\
 \quad \ \  a^T u < e y^*-h^*,  \\
 \quad \ \ \   u \ge 0, \\
\end{array}%
\right. \label{eq170}
\end{equation}
has no solutions with respect to $u$. An arbitrary $y^*\in Y$ with
such a property  together with a solution $x^*$ of the system of
linear inequalities

\begin{equation}\label{pbeda}
\left\{%
\begin{array}{ll}
    A x \le a,\\
    x ^T(Cy^*+g^T) \le h^*-e y^*,  \\
    x \ge 0, \\
\end{array}%
\right.
\end{equation}
with respect to $x$, represent an optimal solution ($x^*,y^*$) for
disjoint bilinear programming problem (\ref{eq1})-(\ref{eq3}).
M<oreover, $h^*$ represents the maximal value of parameter $h$ for
which system (\ref{eq17}) is consistent with respect to $u$ for
every $y\in Y$ and at the same time $h^*$ represents the minimal
value of parameter $h$ in system (\ref{eq14}) for which there exists
$y=y^*$  such that system (\ref{eq14}) is inconsistent with respect
to $u$.
\end{theorem}
\begin{proof} System (\ref{eq013}) is inconsistent if and only if
the system
\begin{equation}
\left\{%
\begin{array}{ll}
    A x \le a,\\
    x ^T(Cy+g^T) < h-e y,  \\
    x \ge 0, \\
\end{array}%
\right. \label{eq015}
\end{equation}
is inconsistent with respect to $x$ for every $y\in Y$. Taking into
account that the set of solutions of system (\ref{eq2}) is nonempty
and bounded we can replace (\ref{eq015}) by the following
homogeneous system
\begin{equation}
\left\{%
\begin{array}{ll}
    A x -at\le 0,\\
    x^T (Cy+g^T) -(ey-h)t< 0,  \\
    x \ge 0, t\geq 0, \\
\end{array}%
\right. \label{eq016}
\end{equation}
preserving the inconsistency property with respect to $x$ and $t$
for every $y\in Y$. Therefore  system (\ref{eq013}) is inconsistent
if and only if system (\ref{eq016}) is inconsistent with respect to
$x$ and $t$ for every $y\in Y$. Applying the duality principle for
system (\ref{eq016}) we obtain that it is inconsistent with respect
to $x$ and $t$ for every $y\in Y$ if and only if system (\ref{eq17})
is consistent with respect to $u$ for every $y$ satisfying
(\ref{eq3}).  Based on this property we may conclude that the
maximal value $h^*$ of parameter $h$ in system (\ref{eq17}) for
which it has solutions with respect to $u$  for every $y\in Y$ is
equal to the minimal value of the objective function in problem
(\ref{eq1})-(\ref{eq3}). At the same time $h^*$, according to
Corollary \ref{corol01}, represents  the minimal value of parameter
$h$ in system (\ref{eq14}) for which there exists $y=y^*$  such that
system (\ref{eq14}) is inconsistent with respect to $u$.

Now let us prove property 1) from the theorem. Assume that system
(\ref{eq140}) is inconsistent with respect to $u$ and $y$. Then the
system
$$\left\{%
\begin{array}{ll}
-{A}^T u - Cy \ - g^Tt \leq 0, \\
\quad \ \  a^T u - e y  + h^*t<0,  \\
\quad \ \ \  u \ge 0, \ \ \ \ \ \ \ \ \ t\geq 0,\\
\end{array}%
\right.$$ is inconsistent. This  involves that system (\ref{eq171})
has solutions. Then for a solution $x^*$ of system (\ref{eq171}) we
have
$${x^*}^TCy+gx^*+ey= y^T(C^Tx^*+e^T)+ gx^*= gx^*=h^*,$$
where $h^*$ is the minimal value of the objective function in
problem (\ref{eq1})-(\ref{eq3}). This means that $x^*$ together with
an arbitrary $y\in Y$ determine an optimal solution $(x^*,y)$ of
problem (\ref{eq1})-(\ref{eq3}). Moreover, in this case the optimal
value $h^*$ of the objective function of the problem does not depend
on the constraints (\ref{eq3}) that define $Y$, i.e. $Y$ may be an
arbitrary subset from $R^m$.

The proof of property 2) of this theorem can be derived from
Corollary \ref{corol01} of Lemma \ref{thm0} and presented above
properties of solutions of system (\ref{eq17}). Indeed, if $h^*$ is
the maximal value of parameter $h$ for which system (\ref{eq17}) has
solutions with respect to $u$ for every $y\in Y$, then according to
Corollary \ref{corol01},  $h^*$ can be treated as the minimal value
of parameter $h$ in (\ref{eq14}) for which there exists $y=y^*\in Y$
such that system (\ref{eq170}) has no solutions with respect to $u$.
This means that system (\ref{eq170}) is inconsistent and the
corresponding homogeneous system
$$\left\{%
\begin{array}{ll}
    -{A}^T u - (Cy^*+g^T)t \ \le 0, \\
 \quad \ \  a u - (e y^*-h^*)t<0;  \\
 \quad \ \ \   u \ge 0,\ \ \ \ \ \ \ \ \ \ \ t\geq 0,\\
\end{array}%
\right. $$ with respect to $u$ and $t$ has no solutions. By applying
duality principle to this system we obtain that  system
(\ref{pbeda}) has solutions with respect to $x$. This means  that
$y^*\in Y$ together with a solution $x^*$ of system (\ref{pbeda})
determine an optimal solution $(x^*,y^*)$ of problem
(\ref{eq1})-(\ref{eq3}).
\end{proof}
\begin{corollary}\label{corr00}
The
linear programming problem:\\
Maximize
 \begin{equation}\label{cr}
 z=h
 \end{equation}
 subject to
\begin{equation}\label{cr1}
\left\{%
\begin{array}{ll}
-{A}^T u - Cy\le g^T, \\
\quad \  \ a u - e y \leq -h,  \\
\quad \ \ \  u \ge 0, \\
\end{array}%
\right. \end{equation}
has solutions  if and only if the  linear programming problem :\\
Minimize
\begin{equation}\label{crd}
 z'=gx
 \end{equation}
subject to
\begin{equation}
\left\{%
\begin{array}{ll}
    \ \ Ax  \le a, \\
    C^Tx=-e^T ,  \\
   \quad \ x  \ge 0, \\
\end{array}%
\right. \label{eq1710}
\end{equation}
has solutions. If $h^*$ is the optimal value of the objective
function of linear programming problem (\ref{cr}),(\ref{cr1}), then
an optimal solution $x^*$ of linear programming problem (\ref{crd}),
(\ref{eq1710}) together with an arbitrary $y\in Y$  is an optimal
solution of the disjoint bilinear  programming  problem  with an
arbitrary subset $Y\in R^m$.
\end{corollary}
\begin{proof} If we dualize linear programming (\ref{cr}), (\ref{cr1} ), then we obtain the problem:\\
\emph{Minimize:}
$$ z'=gx$$
\emph{ subject to }
 $$\left\{%
\begin{array}{ll}
    \ \ -Ax  + \ at \ge 0; \\
    -C^Tx-e^Tt\ge 0 ;  \\
   \quad\quad\quad\quad\quad t=1;\\
   \quad \ x  \ge 0,
\end{array}%
\right.
 $$
 i.e. this problem is equivalent to problem (\ref{crd}), (\ref{eq1710}). So,  problem (\ref{cr}), (\ref{cr1}) has solutions if and only if
 problem (\ref{crd}), (\ref{eq1710}) has solutions.
\end{proof}

If in Lemma \ref{thm0} we take into account that  the set of
solutions \ $X$ \ of system  (\ref{eq2}) is  nonempty and bounded,
then  we obtain the following result.
\begin{corollary}\label{dimop}
 Linear programming problem (\ref{cr}), (\ref{cr1}) has solutions if and only if system (\ref{eq1710}) is consistent.
  If system  (\ref{eq1710}) is inconsistent then the objective function  (\ref{cr})  is unbounded on the set of feasible solutions  (\ref{cr1}).
 \end{corollary}
\begin{remark}\label{remcor1} Let \ $U_h(y)$ \ be the set of solutions of system
(\ref{eq14}) with respect to $u$  for fixed \ $h\in R^1$ and $y\in
R^m$. Additionally, let $\overline{U}_h(y)$ be the set of solutions
of system (\ref{eq17})  with respect to $u$ for  fixed $h\in R^1$
and fixed $y\in R^m$ and denote
 $$Y_h=\{y\in R^m | U_h(y)\neq \emptyset\}; \ \ \  \overline{Y}_h=\{y\in R^m | \overline{U}_h(y)\neq \emptyset\}. $$
In terms of these sets we can formulate the results above as
follows:

\medskip
1. For a given $h$ \ system (\ref{eq13})  has solutions if and only
if \ $Y\not\subset Y_h$ \ and \ the

minimal value $h^*$  of  $h$ for which this  property holds is equal
to the optimal

value of the objective function of problem \
(\ref{eq1})-(\ref{eq3});

\medskip
2. For a given $h$  system \ (\ref{eq013}) \  has no solutions if
and \ only  if \ $Y\subseteq \overline {Y}_h$  and

the maximal value \ $h^*$ \ of \ $h$ \ for which this property holds
is equal to \ the

optimal value of the objective function of problem
(\ref{eq1})-(\ref{eq3});

\medskip
3. If   system (\ref{eq140})  is consistent, then \  $Y\subset
\overline{Y}_h$ \  for  $h< h^*$ \ and \ $Y\not\subset Y_h$ \ for

 $h\geq h^*$, \  i.e, \  $\overline{Y}_{h^*}\setminus
\ Y_{h^*}$ \ represents  the  set \ of \ optimal\ points \ $y^*\in
Y$ \ for

problem  \ (\ref{eq1})-(\ref{eq3});

\medskip
4. If system (\ref{eq140}) is inconsistent,   then
$Y_{h^*}=\emptyset$ and an arbitrary solution \  $x^*$

of \ system \ (\ref{eq171})\  together \ with \ an  arbitrary \
$y\in Y$ \ represent \ an \ optimal

solution of problem (\ref{eq1})-(\ref{eq3}), i.e   $Y_h=\emptyset$
for $h\geq h^*$ and $Y\subset \overline{Y}_h$  for  $h<h^*$.
\end{remark}
\begin{corollary} For a given $h\in [-2^L,2^L]$ system (\ref{eq13})  has
solutions if and only if $Y\not\subset Y_h$; if $Y\subset Y_h$, then
system (\ref{eq13}) has no solutions.
\end{corollary}

\medskip
  Thus, based on Theorem \ref{tm01}, we can replace  disjoint bilinear programming problem
(\ref{eq1})-(\ref{eq3}) with the problem of determining the optimal
value $h^*$ of $h$ and vector $y^*\in Y$ for which  system
(\ref{eq17}) is consistent with respect to $u$. System (\ref{eq17})
can be regarded as a parametric system with right-hand members that
depend on the vector of parameters $y\in Y$ and $h \in R^1$. If in
(\ref{eq13}) we regard $x$ as a vector of parameters, then we can
prove a variant of Theorem \ref{tm01} in which parametrical system
(\ref{eq17}) is replaced by the parametrical system
\begin{equation}
\left\{%
\begin{array}{ll}
    -B^T v \le C^T x +e^T, \\
 \quad \ \   b v \le g x - h, \\
 \quad \ \ \ v \ge 0,\\
\end{array}%
\right. \label{eq18}
\end{equation}
where \ $v$ \ is the vector of variables and \ $x$ \ is the vector
of parameters that satisfies (\ref{eq2}). This means that for the
considered parametric linear systems (\ref{eq17}) and (\ref{eq18})
we can formulate  the following duality principle (see \cite{l87}):
\begin{theorem}
The system of linear inequalities \ (\ref{eq17}) \ is consistent
with respect to $u$ for every $y$ satisfying (\ref{eq3}) if and only
if the system of linear inequalities (\ref{eq18}) is consistent with
respect to $v$ for every $x$ satisfying (\ref{eq2}).
\end{theorem}

All results formulated and proved in this section are valid also for
the case when $Y$ is determined by an arbitrary consistent system
\begin{equation}\label{arb1}
Dy\leq d,
\end{equation}
where $D=(d_{ij})$ is a $p\times m$-matrix and $d$ is a column
vector with $p$ components. In the case when system (\ref{eq3}) is
replaced by system (\ref{arb1}), the following duality principle
holds.
\begin{theorem} The system of linear inequalities \ (\ref{eq17}) \ is consistent
with respect to $u$ for every $y$ satisfying (\ref{arb1}) if and
only if the  system
\begin{equation}
\left\{%
\begin{array}{ll}
    -D^T v = C^T x +e^T, \\
 \quad \ \   d v \le g x - h, \\
 \quad \ \ \ v \ge 0,\\
\end{array}%
\right. \label{eq181}
\end{equation}
is consistent with respect to $v$ for every $x$ satisfying
(\ref{eq2}).
\end{theorem}
\section{The approach for solving the problem  based  on \\Theorem \ref{tm01}}\label{sybsect1}
 The general scheme of the approach we shall use  for solving the disjoint bilinear programming
problem  is based on Theorem \ref{tm01}\ and is as follows:

We replace problem (\ref{eq1})-(\ref{eq3}) \ by  problem of
determining the maximal value \ $h^*$ \ of parameter \ $h$ \ such
that system \ (\ref{eq17}) \ is consistent with respect to \ $u$ \
for every $y$ \ satisfying \ (\ref{eq3}). Then we show how to
determine   the corresponding point \ $x^*$, that satisfy the
conditions of Theorem \ref{thm0}. To apply this approach it is
necessary to develop algorithms for checking conditions 1) and 2) of
Theorem \ref{tm01}, i.e. it is necessary to develop algorithms for
checking if the condition \ $Y\not\subset Y_h$ holds  for a given
$h\in [-2^L,2^L]$.  In general case, the problem of checking such a
condition is a difficult problem from computational point of view,
however for some special cases suitable algorithms can be developed.
In this article we show that if  $Y$ has a structure of an
acute-angled polytope or perfect polytope, then the condition \
$Y\not\subset Y_h$ can be checked in polynomial time. Based on this
we show how to determine the optimal  solution of the disjoint
bilinear programming problem when one of the disjoint subsets has a
structure of an acute-angled polytope. To do this, in Section
\ref{trusc2} we prove Theorem \ref{thm101}, that represents a
refinement of Theorem \ref{tm01} for the case when $Y$ has a
structure of a perfect polytope, and based on this, we show how to
check if $Y\not\subset Y_h$.

\medskip
In general,  if  $Y$  is a bounded set then  condition $Y\not\subset
Y_{h}$  can be detected in finite time by checking the consistency
(or inconsistency) of system  (\ref{eq14}) for all  extreme points
of $ Y$. In analogous way condition   $Y\subseteq \overline{Y}_{h}$
can be verified by checking the consistency of system  (\ref{eq17})
with respect to  $u$  for each extreme point of $Y$.  It is evident
that such an approach for checking condition \ $Y\subseteq Y_{h}$  \
(or $Y\not\subset Y_{h})$) \ is complicate from computational point
of view. The algorithm we propose for the case when $Y$ is a perfect
polytope, avoid exhaustive search. Moreover, we show that in the
case of problems \ (\ref{equat3})-(\ref{equat5}) \ and \
(\ref{eq9})-(\ref{eq12}) \ our approach allows to elaborate
efficient algorithms for solving them.
\section{Some auxiliary results related to redundant inequalities for linear
systems}\label{rdund1} In what follows we shall use some properties
of redundant inequalities for linear systems.
 An inequality
\begin{equation}
\sum_{j=1}^m s_j y_j \le s_0 \label{eq19}
\end{equation}
is called \emph{redundant} for a consistent system of linear
inequalities
\begin{equation}
\sum_{j=1}^m d_{ij} y_j \le d_{i0}, \ \ \ \ \ \ i=1,2,\dots,p,
\label{eq20}
\end{equation}
if (\ref{eq19}) holds for an arbitrary solution of system
(\ref{eq20}). We call the redundant inequality (\ref{eq19})
\emph{non-degenerate} if $s_j\neq 0$ at least for an index
$j\in\{1,2,\dots,m\}$. If  $s_j=0,\ j=1,2,\dots,m $, and \ $s_0\geq
0$ we say that the redundant inequality (\ref{eq19}) is
\emph{degenerate}. We call the redundant inequality (\ref{eq19}) \
for consistent system (\ref{eq20})  \emph{strongly redundant } if
there exists $\epsilon >0$ such that the corresponding inequality
$$\sum_{j=1}^m s_j y_j \le s_0 -\epsilon$$
is redundant for (\ref{eq20}); if  such an  $\epsilon$  does not
exist then we call  inequality (\ref{eq19}) \emph{weakly redundant.}
If an inequality
 \begin{equation}\label{lbl19}
\sum_{j=1}^m d_{kj} y_j \le d_{k0}
 \end{equation} of system
(\ref{eq20}) can be omitted without changing the set of its feasible
solutions then we say that it is redundant in (\ref{eq20}), i.e.,
inequality (\ref{lbl19}) is redundant in (\ref{eq20}) if it is
redundant for the system of the rest of its inequalities.

 The redundancy property of linear inequality (\ref{eq19}) for
consistent system (\ref{eq20}) can be checked based on following
Farkas theorem \cite{che,far}:
\begin{theorem}\label{clasmft}
Inequality (\ref{eq19}) is redundant for consistent system
(\ref{eq20}) if and only if the \ system
\begin{equation}
\left\{%
\begin{array}{ll}
s_j=\sum\limits_{i=1}^p d_{ij} v_i, \ \ \ \ \ \ j=1,2,\dots,m; \\
s_0= \sum\limits_{i=1}^p d_{i0} v_i + v_0; \\
\quad\quad\quad\quad \ \ v_i \ge 0, \ i=0,1,2,\dots,p,
\end{array}%
\right. \label{eq22}
\end{equation}
has solutions with respect to $v_0,v_1,v_2,\dots,v_p.$ Moreover, if
inequality  (\ref{eq19}) is redundant for  system (\ref{eq20}), then
system (\ref{eq22}) has a basic feasible solution $v_0^*,
v_1^*,v_2^*,\dots,v_p^*$ that satisfies the following conditions:

1) the set of column vectors \vspace*{2mm}
$$
\bigg\{D_i=\left(\begin{array}{c}
 d_{i1}\\[1.5mm]
 d_{i2}\\[1.5mm]
  \vdots\\
 d_{im}
 \end{array}
 \right):\ v_i^*>0,\ i\in\{1,2,\dots,p\}\bigg\} $$

$\ \ $ is linearly independent;

2) inequality \ (\ref{eq19}) \ is redundant for the  subsystem of
system \ (\ref{eq19}) \ induced

$\ \ $ by the inequalities that correspond to indices
$i\in\{1,2,\dots,p\}$  with $v_i^*>0$;

3) if \ $v_0^*>0$,   then  inequality   (\ref{eq19})  is strongly
redundant for  system (\ref{eq20}) and

 $\ \ \ $ if  $v_0^*=0$,  then  inequality   (\ref{eq19})  is weakly redundant for  system (\ref{eq20}).
\end{theorem}

The proof of this theorem can be found in  \cite{karm,che,far}.

\begin{corollary}\label{mfcorl1}
Let  redundant inequality (\ref{eq19}) for consistent system
(\ref{eq20}) be given and consider the following linear programming problem:\\
Minimize
\begin{equation}\label{mfeq20}
 z=\sum\limits_{i=1}^p d_{i0} v_i
\end{equation}
subject to
\begin{equation}\label{mfeq21}
\left\{%
\begin{array}{ll}
s_j=\sum\limits_{i=1}^p d_{ij} v_i, \ \ \ \ \ \ j=1,2,\dots,m; \\
\quad \quad\quad\quad \ \ v_i\ge 0, \ \ i=1,2,\dots,p.
\end{array}%
\right.
\end{equation}
 Then  this problem has an optimal solution $v_1^*,v_2^*,\dots,v_p^*$  where $s_0\geq\sum\limits_{i=1}^p d_{i0} v_i^*$. If
$s_0>\sum\limits_{i=1}^p d_{i0} v_i^*$   then inequality
(\ref{eq19}) is strongly redundant for  system (\ref{eq20}) and if
$s_0=\sum\limits_{i=1}^p d_{i0} v_i^*$  then inequality (\ref{eq19})
is weakly redundant for system (\ref{eq20}).
\end{corollary}

 Theorem \ref{clasmft}  can be extended for the case when system
(\ref{eq20}) is inconsistent.
\begin{definition}
 Assume that system (\ref{eq20}) is
inconsistent. Inequality (\ref{eq19}) is called redundant  for
inconsistent system (\ref{eq20}) if there exists a consistent
subsystem
\begin{equation}\label{eq21}
\sum_{j=1}^m d_{i_kj} y_j \le d_{i_k0}, \ k=1,2,\dots,p' (p'<p),
\end{equation}
of system (\ref{eq20}) such that inequality (\ref{eq19}) is
redundant for (\ref{eq21}).
\end{definition}
\begin{theorem}\label{minfac}
Inequality (\ref{eq19}) is redundant for inconsistent system
(\ref{eq20}) if and only if system (\ref{eq22}) has a basic feasible
solution $v_0, v_1, v_2, \dots, v_p$ for which the set of column
vectors
\begin{equation}\label{crll1}
D^+=\Bigg\{\left(\begin{array}{c}
 d_{i1}\\[0.5mm]
 d_{i2}\\[0.5mm]
  \vdots\\[0.5mm]
 d_{im}
 \end{array}
 \right):\ v_i>0,\ i\in\{1,2,\dots,p\} \Bigg\}
 \end{equation}
is linearly independent.  Moreover,  the  subsystem of inconsistent
system (\ref{eq20}) induced
 by inequalities that correspond to indices  with
$v_i^0>0$ is a consistent subsystem of  system (\ref{eq20}).
\end{theorem}
\begin{proof}
$\Rightarrow$ Assume that inequality \ (\ref{eq19}) \ is redundant
for  inconsistent system \ (\ref{eq20}).  Then there exists a
consistent subsystem \ (\ref{eq21}) \ of system \ (\ref{eq20}) \
such that \ (\ref{eq19}) \ is redundant for \ (\ref{eq21}). Then
according to Theorem \ref{clasmft} there exists a basic feasible
solution $v_0, v_{i_1}, v_{i_2}, \dots, v_{i_{p'}}$ for the system
$$\left\{%
\begin{array}{ll}
s_j=\sum\limits_{k=1}^{p'} d_{i_kj} v_{i_k}, \ \ \ \ \ \ \ \ \ \quad \quad j=1,2,\dots,m; \\
s_0= \sum\limits_{k=1}^{p'} d_{i_k0} v_{i_k} + \ v_0; \\
\quad\quad\quad\quad \  v_0\ge 0, \ v_{i_k} \ge 0, \ \
k=1,2,\dots,p',
\end{array}%
\right.$$ such that the set of column vectors
$$
\bigg\{d_{i_k}=\left(\begin{array}{c}
 d_{i_k1}\\[1.5mm]
 d_{i_k2}\\[1.5mm]
  \vdots\\
 d_{i_km}
 \end{array}
 \right):\ v_{i_k}>0, \ k\in\{1,2,\dots,p'\}
 \bigg\} $$
is linearly independent.

$\Leftarrow$ Let \ (\ref{eq20}) \ \ be an arbitrary inconsistent \
system \ and \ \ $v_0, v_1, v_2, \dots,v_p$ \ be \ a \ solution of
system \ (\ref{eq22}) \ that contains \ $p'\geq 1$ \ nonzero
components \ $ v_{i_1}, \ v_{i_2}, \ \dots, \ v_{i_{p'} } \ $ \ such
that \ the corresponding system of column vectors $\{d_{i_k}:\
v_{i_k}>0, \ k=1,2,\dots,p' \}$ is linearly independent. Then
$p'\leq \min \{m,p\}$ and the corresponding system $$\sum_{j=1}^m
d_{i_kj} y_j \le d_{i_k0}, \ \ \ \ \ \ k=1,2,\dots,p',$$ has
solutions. Based on Theorem \ref{clasmft} we obtain that inequality
(\ref{eq19}) is redundant for system (\ref{eq21}). This means that
inequality (\ref{eq19}) is redundant for inconsistent system
(\ref{eq20}).
\end{proof}
\section{The optimality criterion for the  disjoint bilinear programming problem with a perfect disjoint subset}\label{trusc2}
 In this section we present a refinement of Theorem \ref{tm01} for the
disjoint bilinear programming
 problem  (\ref{eq08})-(\ref{eq38}) with conditions $a)-c)$ for system
 (\ref{eq38}). In fact, this refinement  is related to the case when $Y=\{y | By\leq b, \ y\geq 0\}$ is replaced by  $Y=\{y | Dy\leq
 d\}$ that satisfies conditions  $a)-c)$. We show that that in this case the
 optimality criterion for problem  (\ref{eq08})-(\ref{eq38}) with conditions
 $a)-c)$ can be formulated in  terms of the existence of a basic
 solution with the given basic component for a system of
 linear  equations with nonnegative conditions for the variables. We present the optimality criterion in new terms for  problem (\ref{eq08})-(\ref{eq38}) in the following  extended form:\\
 \noindent \emph{Minimize}
\begin{equation}\label{eq01}
z=\sum\limits_{i=1}^{n}\sum\limits_{j=1}^{m}c_{ij}x_iy_j+\sum\limits_{i=1}^ng_ix_i+\sum\limits_{j=1}^me_jy_j
\end{equation}
\emph{subject to}
\begin{equation}\label{eq02}
\left\{%
\begin{array}{ll}
 \displaystyle\sum\limits_{j=1}^n a_{ij}x_j \le a_{i0},\ \ \ i=1,2,\dots,q,\\[0.5mm]
\quad\quad \ \ x_j \ge 0, \ \ \ \ \ j=1,2,\dots,n;
\end{array}%
\right.
\end{equation}
\begin{equation}
\qquad\quad\quad \ \ \ \sum\limits_{j=1}^md_{ij}y_j\le d_{i0}, \ \ \
i=1,2,\dots,p \ \ (m<p), \label{eq03}
\end{equation}

\medskip So, we will assume that the set of solutions $Y$ of system
(\ref{eq03}) in this problem  satisfies the following conditions:

\medskip\noindent
 $\ \ \ \    $a) \emph{system (\ref{eq03}) does not
contain redundant inequalities and the set \ of \ its}

$\ \   $ \emph{solutions  $Y$ is  a  bounded  set with nonempty
interior;}

\medskip\noindent
$\ \ \ \ b)$ \emph{the set  of solutions of an arbitrary  subsystem}
$$\sum_{j=1}^m d_{i_kj} y_j \le d_{i_k0}, \ \ \ \ k=1,2,\dots,m,$$
$\ \ \ \ \ \ \ $ \emph{of \  rank \ $m$ \ represents a convex \ cone
\ $Y^-(y^r)$ \ with the origin \ at \ an}

$\ \  \ $\emph{extreme point $y^r$ from the set of extreme points
$\{y^1,y^2,\dots,y^N\}$ \ of \ the}

$\ \ \ $\emph{set of solutions $Y$ of system (\ref{eq03});}

\medskip\noindent
$\ \ \      $ c) \emph{at each extreme point \
$y^r\in\{y^1,y^2,\dots,y^N\}$ \ of the set \ of \ solutions \ $Y$}

$\ \  $ \emph{of \ system \ (\ref{eq03}) \ exactly \ $m$ \
hyperplanes \ of \  the  facets  of \ polytope \ $Y$}

$\ \ $ \emph{intersect.}

\medskip\noindent
It is easy to see that if \ \vspace*{2mm} $\arraycolsep 4pt D =\left
( \begin{array}{lll}
 \ B  \\[0.5mm]
 -I  \\
\end{array}\right)
\arraycolsep 5pt ,$ $\arraycolsep 4pt d =\left (
\begin{array}{lll}
  b  \\[0.5mm]
  0  \\
\end{array}\right)
\arraycolsep 5pt $,  $I$ is the identity matrix and \ $0$ \ is the
column vector with zero components, then problem
(\ref{eq01})-(\ref{eq03}) becomes problem (\ref{eq1})-(\ref{eq3}).
Additionally if matrix $B$ is an identity one or step-diagonal then
we obtain a disjoint bilinear programming problems for which
conditions $a)-c)$ hold.

The main results we describe in this section are concerned  with
elaboration of an algorithm that determines if the property $Y \not
\subset Y_h$ holds.
\subsection{The properties of the extreme points for the set of solutions of system (\ref{eq03})}
 Let $y^r=(y_1^r,y_2^r, \dots, y_m^r),  r=1,2,\dots N$  be the extreme points of the set of solutions $Y$ of
system (\ref{eq03}) that satisfies conditions  $a)-c)$. Then for
each $y^r, r\in\{1,2,\dots,N\}$, there exists a unique subsystem
\begin{equation}\label{dscheck3}
\sum_{j=1}^m d_{i_kj} y_j \le d_{i_k0}, \ \ \ \ k=1,2,\dots,m,
\end{equation}
of rank $m$ of system (\ref{eq03}) such that  $y_1^r,y_2^r, \dots,
y_m^r$ is the solution of the system of linear equations
\begin{equation}
\sum_{j=1}^m d_{i_kj} y_j = d_{i_k0}, \ \ \ \ k=1,2,\dots,m.
\end{equation}
Denote
$$I(y^r)= \{i\in\{1,2,\dots, p\} :
\sum\limits_{j=1}^md_{ij}y_j^r=d_{i0}\}$$ and consider the
\emph{convex cone  $Y^-(y^r)$ for system (\ref{eq03})} as the
solution set of the following system
\begin{equation}\label{ldeq01}
\sum\limits_{j=1}^md_{ij}y_j\le d_{i0} ,  \  \ i\in
I(y^r),\end{equation} where  $y^r=(y_1^r,y_2^r, \dots, y_m^r)$  is
the \emph{origin of the cone} $Y^-(y^r)$ and $|I(y^r)|=m$.  We call
the solution set of the system
$$
\sum_{j=1}^m d_{ij} y_j \ge d_{i0}, \ \ i\in I(y^r)
$$
\emph{the symmetrical cone for  $Y^-(y^r)$} and denote it by
$Y^+(y^r)$.

Obviously,  $Y^-(y^r), \ Y^+(y^r)$  represent  convex polyhedral
sets with interior points such that $Y= \cap_{r=1}^N Y^-(y^r)$ and
$Y^+(y^r)\cap Y^-(y^r)=y^r, \ r=1,2,\dots,N$.
\begin{lemma}\label{lemmap01} If \   $Y^+(y^1), \ Y^+(y^2), \  \dots,\  Y^+(y^N) $ \ represent the
symmetrical cones for the corresponding  cones \ $Y^-(y^1), \
Y^-(y^2),\   \dots, \ Y^-(y^N) $ \ of system \  (\ref{eq03}) with
properties $a)-c)$,  then  \ $Y^+(y^r)\bigcap Y^+(y^k)=\emptyset$ \
for $r\neq k$. Additionally if \ $z^1, \ z^2,\ \dots, \ z^N$  \
represent arbitrary points \ of \ the \ corresponding \ cones \
$Y^+(y^1), \ Y^+(y^2),  \dots, Y^+(y^N),$ \ then the convex hull \
$Conv\big(z^1,z^2,\dots,x^N\big) $ \ of points $z^1, \  z^2,  \
\dots, z^N $  contains $ Y$.
\end{lemma}
\begin{proof}  \  The property  $Y^+(y^r)\cap Y^+(y^k)=\emptyset$ for $r\neq k$ can be proven by contradiction. Indeed, if we  assume  that $Y^+(y^r)\cap
Y^+(y^k)\neq\emptyset$,  then this polyhedral set contains an
extreme point $y^0$, where $y^0\not\in Y$, because $Y^+(y^r)\cap
Y^+(y^k)$ is determined by  the system of inequalities consisting of
inequalities that define $Y^+(y^r)$ and inequalities that define
$Y^+(y^k)$. This means that $y^0=(y^0_1,y^0_2,\dots,y^0_m)$ is the
solution for a system of equations
$$\sum_{j=1}^m d_{i_kj} y_j = d_{i_k0}, \ \ \ \ k\in\{1,2,\dots,m\},$$
of rank $m$. Then according to  properties $a)-c)$  we obtain that
$y^0$ is a solution of system (\ref{eq03}) which is in contradiction
with the fact that $y^0\not\in Y$. So, $Y^+(y^r)\cap
Y^+(y^k)=\emptyset$ for $r\neq k$.

 Now let us show that if  \ $z^1, \ z^2, \ \dots,
 \ z^N$  \ represent arbitrary points of the corresponding
sets \ $Y^+(y^1),\  Y^+(y^2), \ \dots, \ Y^+(y^N)$, then the convex
hull $Conv(z^1,z^1,z^2,\dots,z^N)$ of points $ z^1, z^1, z^2,\dots,
z^N$ contains $Y$. Indeed, if we construct the convex hull \
$Y^1=Conv(z^1,y^2,\dots,y^N)$ \ of points \ $z^1,y^2,\dots,y^N$,
then \ $y^1\in Y^1$ \ and \ $Y\subseteq Y^1$. \  If\ after that we
construct \ the   convex  hull \ $Y^2=Conv(z^1,z^2,y^3\dots,y^N)$ \
of points \ $(z^1,z^2,y^3,\dots,y^N)$, then \  $y^2\in Y^2$ \ and \
$Y^1\subseteq Y^2$ \ and so on. \ Finally at step \ $N$ \ we
construct the convex hull \ $Y^N=Conv(z^1,z^2,\dots,z^N)$ \ of
points \ $z^1,z^2,\dots,z^N$ \ where $y^N\in Y^N$ and
$Y^{N-1}\subseteq Y^N$, \ i.e.\ $Y\subseteq Y^1\subseteq
Y^2\subseteq \dots \subseteq Y^N=Conv\big(z^1,z^2,\dots,z^N\big)$.
\end{proof}
\begin{corollary} If system (\ref{eq03})  satisfies conditions  $a)-c)$,  then the
system $$\sum\limits_{j=1}^md_{ij}y_j\geq d_{i0},\ \ \ \
i=1,2,\dots,p,$$ is inconsistent and the inequalities of this system
can be divided into N disjoint consistent subsystems
$$\sum\limits_{j=1}^md_{ij}y_j\geq d_{i0},\ \ \ \ i\in I(y^r), \
r=1,2,\dots,N$$ such that $Y^+(y^l)\cap Y^+(y^k)=\emptyset$ for
$l\neq k$.
\end{corollary}
\begin{lemma}\label{somnen}
Assume that system  (\ref{eq03}) satisfies conditions  $a)-c)$. If
an inequality \begin{equation}\label{eq002}-\sum_{j=1}^m s_j y_j \le
-s_0\end{equation} is redundant for the inconsistent system
\begin{equation}
-\sum\limits_{j=1}^md_{ij}y_j\le -d_{i0},\ \ \ \ i=1,2,\dots,p,
\label{eq003}
\end{equation} then such an inequality is redundant at least for
a  consistent subsystem
$$-\sum\limits_{j=1}^md_{ij}y_j\leq -d_{i0},\ \ \ \ i\in I(y^r), \ r\in\{1,2,\dots,N\},$$
of  inconsistent system (\ref{eq003}). \end{lemma}\
\begin{proof} Assume that inequality (\ref{eq002}) is redundant
for inconsistent system (\ref{eq003}). Then there exits a consistent
subsystem of rank $p'$ $$-\sum\limits_{j=1}^md_{i_kj}y_j\leq
-d_{i0},\ \ \ \ k=1,2,\dots p',$$ for inconsistent system
(\ref{eq003}) such that $p'\leq m$ and inequality (\ref{eq002}) is
redundant for this system. If $p'=m$, then the set of solution of
this system will represents a convex cone $Y^+(y^r), r\in
\{1,2,\dots,N\}$, i.e in this case lemma holds. If $p'<m$, then the
set of solutions of this system contains at least a convex cone
$Y^+(y^r), r\in \{1,2,\dots,N\}$, so lemma holds.
\end{proof}
\subsection{A criterion for checking if  $Y \not \subset Y_h$ based on properties  of redundant inequalities for inconsistent systems  of linear inequalities}
In this subsection we present a criterion for checking the condition
\ $Y \not\subset Y_h$ that we shall use in the next subsection for
the proof of the optimality criterion for problem
(\ref{eq01})-(\ref{eq03}). Such a criterion we formulate and prove
by using the results from Section \ref{rdund1} related to redundant
inequalities for an inconsistent system of linear inequalities.

 If for problem (\ref{eq01})-(\ref{eq03}) we consider system
 (\ref{eq17}) (see Theorem  \ref{tm01})
\begin{equation}\label{eq25}
\left\{%
\begin{array}{ll}
\displaystyle -\sum_{i=1}^q a_{ij} u_i - \sum_{j=1}^m c_{ij} y_j \le g_i, \ \ \ \ \ j=1,2,\dots,n, \\
\displaystyle \quad \sum_{i=1}^q a_{i0} u_i - \sum_{j=1}^m e_j y_j \le -h, \\
\quad\quad\quad\quad u_i \ge 0, \ \ \ \ \ \ i=1,2,\dots,q,\\
\end{array}%
\right.
\end{equation}
then  either  this system has solutions with respect to
$u_1,\!u_2,\dots,\!u_q,y_1,y_2,\dots ,\!y_m$   for an arbitrary
$h\in R^1$ or there exists a minimal value $h^*$  of $h$ for which
this system has a solution.  According to Corollary \ref{corr00} if
for system (\ref{eq25}) there exists a minimal value $h^*$ for which
it is consistent, then the optimal solution of problem
(\ref{eq01})-(\ref{eq03}) can be found by solving linear programming
problem (\ref{crd}),(\ref{eq1710}).
 Therefore in that
follows we will analyze  the case when  system (\ref{eq25}) has
solutions for every $h\in R^1$, i. e. the case when system
(\ref{eq1710}) has no solutions.
\begin{lemma} \label{lemma01} Let  disjoint bilinear programming
problem  (\ref{eq01})-(\ref{eq03})\ \ be such that  system
(\ref{eq03})\ satisfies conditions \ $a)- c)$\  and the set of
solutions \ $X$ of system (\ref{eq02}) is nonempty and bounded. If
system (\ref{eq1710}) has no solutions, then for  a given \ $h$\ the
property \ $Y \not \subset Y_h$ \ holds if and only if for
consistent system (\ref{eq25}) there exists a non-degenerate
redundant inequality
\begin{equation}
\sum_{j=1}^m s_j y_j \le s_0 \label{eq001}
\end{equation}
such that the corresponding symmetrical \ inequality \ (\ref{eq002})
 is redundant for the \ inconsistent \ system (\ref{eq003}).
\end{lemma}
\begin{proof}
$\Rightarrow$Assume that  system  (\ref{eq1710})  has no solutions.
$\!$Then according to Corollary \ref{dimop}  of Theorem \ref{tm01}
we have \ $Y_h\neq \emptyset$. Therefore if \ $Y\not\subset Y_h$,
then among the extreme points $y^1,y^2,\dots,y^N$ of $Y$ there
exists at least one that does not belong to $Y_h$. Denote by
$y^1,y^2,\dots,y^{N'}$ the extreme points of $Y$ that do not belong
to $Y_h$ and by $y^{N'+1},y^{N'+2},\dots,y^N$ the extreme points of
$Y$ that belong to $Y_h$. Each extreme point
$y^l=(y_1^l,y_2^l,\dots,y_m^l),$ $ l\in\{1,2,\dots,N\}$,   of $Y$
represents the vertex of the cone $Y^-(y^l)$ that is determined by
the solution set of system (\ref{ldeq01}) for $r=l$. At the same
time each extreme point $y^l$ of $Y$ is the vertex  of the
symmetrical cone $Y^+(y^l)$ that is determined by the solution set
of the subsystem of linear inequalities
\begin{equation}\label{004}
-\sum_{j=1}^m d_{ij} y_j \leq -d_{i0}, \ \ \ \ i\in I(y^l),
\end{equation}
 of  inconsistent system
(\ref{eq003}).  According to Lemma \ref{lemmap01} we have \
$Y^+(y^l)\cap  Y^+(y^k)=\emptyset$ for $l\neq k$.

Let us show that among the extreme points $y^1,y^2,\dots,y^{N'}$
there exists a point $y^{j_0}$ for which the corresponding cone
$Y^+(y^{j_0})$ has no common elements with   $Y_h$,  i.e. \
$Y^+(y^{j_0})\bigcap Y_h=\emptyset $. This fact can be proved  using
the rule of contraries. If we assume that \ $Y^+(y^{r})\bigcap
Y_h\neq\emptyset,$ \ $l=1,2,\dots,N'$, then we can select from each
set $Y^+(y^{r})\bigcap Y_h$ an element $z^r$  and construct the
convex hull \ $Conv(z^1,\ z^2,\ \dots, \ z^{N'}, y^{N'+1},\
y^{N'+2}, \ \dots, \ y^N)$ \ for the set of the points \ $\{z^1,\
z^2, \ \dots,z^{N'},\  y^{N'+1},\ y^{N'+2}, \ \dots, \ y^N\}$.
Taking into account that \ $z^r\in Y_h,$ \ $r=1, \ 2, \ \dots, \ N'$
\ and \ $y^{N'+l}\in Y_h, \ r=1,2,\dots, N-N'$ \ we have \
$Conv(z^1,z^2,\dots,z^{N'}, y^{N'+1},y^{N'+2},\dots,y^N)\subseteq
Y_h$. However according to Lemma \ref{lemmap01}  we have $Y\subseteq
Conv(z^1,z^2,\dots,z^{N'}, y^{N'+1},y^{N'+2},\dots,y^N)$, i.e. we
obtain $Y\subseteq Y_h$. This is in contradiction with the condition
\ $y^r\not\in Y_h$ for $ r=1,2,\dots, N'$.

Thus, among $Y^+(y^{1}),Y^+(y^{2}),\dots,Y^+(y^{N'} )$ there exists
a cone $Y^+(y^{j_0})$ with vertex $y^{j_0}$ for which \
$Y^+(y^{j_0})\bigcap Y_h=\emptyset $. Therefore for convex sets
$Y^+(y^{j_0})$ and $Y_h$ there exists a separating hyperplane
\cite{karm,far,luy,sh1}
$$\sum\limits_{j=1}^m s_j y_j = s_0$$ such that $$\displaystyle
\sum_{j=1}^n s_j y_j < s_0, \quad \forall (y_1,y_2,\dots,y_n) \in
Y_h,$$ and
$$-\sum_{j=1}^m s_j y_j \le -s_0, \quad \forall (y_1,y_2,\dots,y_n) \in \ Y^+(y^{j_0}).$$
So, the inequality \ $\sum\limits_{j=1}^n s_j y_j \le s_0$ \ is
redundant for system (\ref{eq25}) and the inequality \
$-\sum\limits_{j=1}^n s_j y_j \le -s_0$  \ is redundant for
consistent subsystem (\ref{004}) of inconsistent system
(\ref{eq003}).

$\Leftarrow$ Assume that  for system (\ref{eq25}) there exists a
non-degenerate redundant inequality (\ref{eq001}) such that the
corresponding inequality (\ref{eq002}) is redundant for inconsistent
system (\ref{eq003}). Then there exists a consistent subsystem
(\ref{004}) of system (\ref{eq003}) for which the conical subset
$Y^-(y^{i_0})$ has no common points with $Y_h$ where $Y_h\neq
\emptyset$, i.e. $y^{j_0}\not\in Y_h$. Taking into account that
$y^{j_0}\in Y$ we obtain $Y \not \subset Y_h$.
\end{proof}
\begin{corollary} \label{cor01} \ Assume that the conditions of Lemma \ref{lemma01}\ are satisfied. Then  the minimal value $z^*$ of the objective function of
problem (\ref{eq01})-(\ref{eq03}) \ is equal to the minimal value
$h^*$ of parameter $h$ in system (\ref{eq25}) for which there exists
a non-degenerate redundant inequality
\begin{equation}\label{heq1}\sum_{j=1}^m s_j^* y_j \le s_0^* \end{equation}
such that the corresponding symmetrical inequality
\begin{equation}\label{heq2}-\sum_{j=1}^m s_j^* y_j \le -s_0^*\end{equation}
is redundant for inconsistent system (\ref{eq003}). An optimal point
$y^*$ for  problem  (\ref{eq01})-(\ref{eq03}) can be found by
solving the following system
\begin{equation}
\left\{%
\begin{array}{ll}
\displaystyle \quad \sum_{j=1}^m d_{ij} y_j \leq d_{i0}, \ \ \ \ \ \ i=1,2,\dots,p\ , \\
\displaystyle \quad \sum_{j=1}^m s_j^* y_j = s_0^*.
\end{array}%
\right. \label{eq026}
\end{equation}
\end{corollary}
\begin{proof} Assume that for
system (\ref{eq25}) with given $h$ there exists  a non-degenerate
redundant inequality (\ref{eq001}) such that  symmetrical inequality
(\ref{eq002}) is redundant for system (\ref{eq003}). Then according
to  Corollary \ref{corol01} of Lemma \ref{thm0} and Theorem
\ref{tm01} for   problem (\ref{eq01})-(\ref{eq03}) there exists a
feasible solution such that the corresponding value of the objective
function is not greater than $h$. Therefore  the minimal value $z^*$
of the objective function of problem (\ref{eq01})-(\ref{eq03}) \ is
equal to minimal value $h^*$ of parameter $h$ in system (\ref{eq25})
for which there exists a non-degenerate redundant inequality
(\ref{heq1}) such that the corresponding symmetrical inequality
(\ref{heq2}) is redundant for inconsistent system (\ref{eq003}). In
this case $y^{j_0}=y^*\in Y\cap bd(\overline{Y}_{h^*})$ and the
optimal point $y^*\in Y$ can be found by solving system
(\ref{eq026}).
\end{proof}

\subsection{The optimality criterion for the disjoint bilinear programming problem (\ref{eq01})-(\ref{eq03})}\label{ddl08}
Based on results from previous subsection we can prove the following optimality criterion for problem (\ref{eq01})-(\ref{eq03}).
\begin{theorem}\label{thm101}
Let  disjoint bilinear programming problem (\ref{eq01})-(\ref{eq03})
be such that  system (\ref{eq03}) satisfies conditions $a)-c)$ and
the set of solutions $X$ of system (\ref{eq02}) is nonempty and
bounded. If system (\ref{eq1710}) has no solutions, then for a given
$h$ the property $Y \not \subset Y_h$ holds if and only if the
following system
\begin{equation}\label{sure1}
\left\{%
\begin{array}{ll}
\displaystyle \quad \sum_{j=1}^n a_{ij} x_j +x_{n+i}= a_{i0}, \quad\quad\quad\quad\quad\quad \ \ \ \ \ \ i=1,2,\dots,q, \\
\displaystyle \quad \sum_{i=1}^n c_{ij} x_i+\sum_{k=1}^p d_{kj} v_k = -e_j,\quad\quad\quad\quad\  \ \ \ \ \ j=1,2,\dots,m, \\
\displaystyle \quad \sum_{i=1}^n g_i x_i\ -\sum_{k=1}^p d_{k0} v_k + v_{p+1}= h, \\
\quad\quad\ \ \ x_i\ge 0,\ i=0,1,2,\dots,n+q; \ v_k \ge 0,\
k=1,2,\dots,p+1,
\end{array}%
\right.
\end{equation} has a basic feasible solution
$x_1^0,x_2^0,\dots,x_{n+q}^0,v_1^0,v_2^0,\dots,v_{p+1}^0$ \ for
which the set of vectors
\begin{equation}\label{alt1}
\bigg\{D_k=\left(\begin{array}{c}
 d_{k 1}\\[1.5mm]
 d_{k 2}\\[1.5mm]
  \vdots\\
 d_{k m}
\end{array}
 \right):\ v_k^0>0,  \quad k\in\{1,2,\dots,p\}\bigg\}
 \end{equation}
is linearly independent. The minimal value \ $h^*$ \ of parameter \
$h$ \ with the  property  $Y \not \subset Y_h$  is equal to the of
the optimal value of the objective function of problem
(\ref{eq01})-(\ref{eq03}). Moreover, for $h=h^*$ system
(\ref{sure1}) has a basic solution
$x_1^*,x_2^*,\dots,x_{n+q}^*,v_1^*,v_2^*,\dots,v_{p+1}^*$ with the
 basic component $v_{p+1}^*=0$.
\end{theorem}
\begin{proof}
According to Lemma \ref{lemma01}  the condition  $Y \not \subset
Y_h$  holds if and only if there exist $s_0,s_1,s_2,\!\dots,\!s_n$
such that (\ref{eq001}) is  non-degenerate redundant inequality for
system (\ref{eq25}) and (\ref{eq002}) is a redundant inequality for
inconsistent system (\ref{eq003}). If for (\ref{eq001}) and
\ref{eq25}) we apply the Minkowski-Farkas theorem (Theorem
\ref{clasmft}), then we obtain that (\ref{eq001}) is redundant for
(\ref{eq25}) if and only if there exist $x_0, x_1,
x_2,\dots,x_{n+q}, t$ such that
\begin{equation}
\left\{%
\begin{array}{ll}
0=\displaystyle -\sum_{j=1}^n a_{ij} x_j- x_{n+i} + a_{i0}t, \ \ \ \ i=1,2,\dots,q, \\
s_j=\displaystyle -\sum_{i=1}^n c_{ij} x_i - e_jt, \quad\quad\quad \ \ \ \ \ j=1,2,\dots,m, \\
s_0=\quad \displaystyle \sum_{i=1}^n g_i x_i -ht + x_0, \\
\quad\quad\quad\quad\quad\ \    x_j \ge 0, \ i=0,1,2,\dots,n+q; \ t\geq 0.\\
\end{array}%
\right. \label{equ31}
\end{equation}

 Now let us apply  Theorem \ref{minfac} for inequality
(\ref{eq002}) and inconsistent system (\ref{eq003}). According to
this theorem, inequality (\ref{eq002}) is redundant for system
(\ref{eq003}) if and only if there exist  $v_0,v_1,v_2,\dots, v_m$
such that
\bigskip
\begin{equation}
\left\{%
\begin{array}{ll}
\displaystyle -s_j=-\sum_{k=1}^p d_{kj} v_k, \quad\quad\ \ j=1,2,\dots,m, \\
\displaystyle -s_0=-\sum_{k=1}^n d_{k0} v_k + v_0, \\
\quad\quad\quad\quad\quad\quad\ \ v_k \ge 0, \ \ k=0,1,2,\dots,p\ ,
\end{array}%
\right. \label{eq32}
\end{equation}
where  \ $\sum_{k=1}^p d_{kj} v_i\neq 0$ \  at  least  for an index
\ $j\in \{1,2,\dots,m\}$ \ and the nonempty set of column vectors
(\ref{alt1}) \ is linearly independent.

So,  Lemma \ref{lemma01} can  be  formulated  in  terms  of
solutions of systems (\ref{equ31}), (\ref{eq32}) as follows: the
condition \ $Y \not \subset Y_h$\ holds if and only if there exist \
$s_0$, $s_1$, $s_2, \dots,s_m$, \ $ x_0, x_1, x_2,\ \dots,\
x_{n+q}$, \ $v_0, v_1,v_2,\ \dots, \ v_{p+1}, t$ \ that satisfy \
(\ref{equ31}),\ (\ref{eq32}), where \ $\sum_{k=1}^p d_{kj} v_k \neq
0$ \ at least for an index $j \in \{1,2,\dots,m\}$ and the set of
column vectors (\ref{alt1}) is linearly independent. If we eliminate
$s_1, s_2, \dots, s_m$ from (\ref{equ31}) by introducing
(\ref{eq32}) in (\ref{equ31}) and after that denote $v_{p+1}=v_0+x_0
$, then we obtain the following system
 \begin{equation}
\left\{%
\begin{array}{ll}
\displaystyle \sum_{j=1}^n a_{ij}x_j+x_{n+i} =a_{i0}t, \ \quad\quad \quad i=1,2,\dots,q, \\
\displaystyle \sum_{i=1}^n c_{ij} x_i+\sum_{k=1}^p d_{kj} v_k = -e_jt, \ \ \ \ j=1,2,\dots,m, \\
\displaystyle \sum_{i=1}^n g_i x_i\ -\sum_{k=1}^p d_{k0} v_k +v_{p+1}= ht; \\
\quad\quad \  x_j \ge 0, \ j=1,2,\dots,n+q;\ v_k \ge 0,  k=1,2,\dots,p+1, \ t\geq 0.\\
\end{array}%
\right. \label{eq290}
\end{equation}
This means that Lemma \ref{lemma01}\ \  in terms of solutions of
system \ (\ref{eq290}) \ can be formulated \ as \ follows: the
property \ $Y \not \subset Y_h$ \ holds if and only if \ system \
(\ref{eq290}) \ has \ a \ solution  \ $ x_1, \ x_2, \
\dots,x_{n+q},$ \ $v_1, \ v_2 \ ,\dots,\ v_p, \ v_{p+1},\ t$ \ where
\ $\sum_{k=1}^p d_{kj} v_i \neq 0$ \ at least for an index \ $j \in
\{1,2,\dots,m\}$ \ and the nonempty set of column vectors
(\ref{alt1}) is linearly independent. In (\ref{eq290}) the subsystem
\begin{equation}\label{lo1}
\left\{%
\begin{array}{ll}
\displaystyle \quad \sum_{j=1}^n a_{ij} x_j+x_{n+i} =a_{i0}t, \ \quad\quad \quad i=1,2,\dots,q, \\
\quad\quad\quad \  x_j \ge 0, \ j=1,2,\dots,n+q, t\geq 0,\\
\end{array}%
\right. \end{equation} has  a nonzero solution  \ $ x_1, x_2,\
\dots,\ x_{n+q}, t $ \  if and only if \ $t>0$,\  because the set of
solutions of system \ (\ref{eq02}) is nonempty and bounded. Based in
this and taking into account that we are seeking for a solution of
system (\ref{eq290}) such that  $\sum_{k=1}^p d_{kj} v_k^0 \neq 0$ \
at least for an index \ $j \in \{1,2,\dots,m\}$,   we can set \
$t=1$ \ in (\ref{eq290} )\ and therefore finally obtain that $Y \not
\subset Y_h$ if and only if system \ (\ref{sure1})\ has a solution \
$x_1^0,\ x_2^0, \ \dots, \ x_{n+q}^0, \ v_1^0, \ v_2^0, \ \dots, \
v_{p+1}^0$ \ \ for which the set of column vectors (\ref{alt1}) is
linearly independent and \ $\sum_{k=1}^p d_{kj} v_k^0 \neq 0$ \ \ at
least for an index \ $j \in \{1,2,\dots,m\}$.

 If $h^*$ is the minimal value of parameter $h$ for
which system \ (\ref{sure1}) \ has a  basic feasible solution  where
the system of vectors (\ref{alt1}) is linearly independent, then
according to Corollary \ref{cor01} of Lemma \ref{lemma01},  $h^*$ is
equal to the optimal value of the objective function of problem
(\ref{eq01})-(\ref{eq03}). Obviously, if $h=h^*$, then system
(\ref{sure1}) has a basic solution
$x_1^*,x_2^*,\dots,x_{n+q}^*,v_1^*,v_2^*,\dots,v_{p+1}^*$ with the
degenerate basic component $v_{p+1}^*=0$.
 \end{proof}
\begin{corollary}\label{blin0}
Let   disjoint bilinear programming problem
(\ref{eq01})-(\ref{eq03}) be such that system (\ref{eq03}) satisfies
conditions $a)-c)$ and the set of solutions $X$ of system
(\ref{eq02}) is nonempty and bounded. If system (\ref{eq1710}) has
no solutions, then for a given  $h$ the property $Y \not \subset
Y_h$ holds if and only if  system (\ref{sure1}) has a basic feasible
solution \
$x_1^0,x_2^0,\dots,x_{n+q}^0,v_1^0,v_2^0,\dots,v_{p+1}^0$, where
$v_{p+1}^0$ is a basic component.  The minimal value \ $h^*$ \ of
parameter \ $h$ \ with such a property is equal to the  optimal
value of objective function of problem (\ref{eq01})-(\ref{eq03}).
\end{corollary}
\begin{proof} In  Theorem \ref{thm101} the  condition of the existence
of a basic feasible solution $x_1^0,x_2^0,\dots,x_{n+q}^0,$ \
$v_1^0,v_2^0,\dots,v_{p+1}^0$  for system (\ref{sure1}) such that
the set of vectors (\ref{alt1}) is linearly independent is
equivalent to the condition of the existence for this system a basic
feasible solution $x_1^0,x_2^0,\dots,x_{n+q}^0,$ \
$v_1^0,v_2^0,\dots,v_{p+1}^0$ with the basic component $v_{p+1}^0$.
Therefore the corollary holds.
\end{proof}
\begin{remark} Theorem \ref{thm101}  holds also when for a given $h$ system (\ref{sure1})
has a basic feasible solution
$x_1^0,x_2^0,\dots,x_{n+q}^0,v_1^0,v_2^0,\dots,v_{p+1}^0$ for which
the set of vectors (\ref{alt1}) is an empty set because, based on
Corollary \ref{corr00}, we obtain that
$x^*=(x_1^0,x_2^0,\dots,x_n^0)^T$ is an optimal point for the the
disjoint bilinear programming problem (\ref{eq01})-(\ref{eq03}). In
this case $x^*$ together with an arbitrary $y\in Y$ ia an optimal
solution of problem (\ref{eq01})-(\ref{eq03}).
\end{remark}

\begin{remark} Theorem \ref{thm101} is valid for problem  (\ref{eq01})-(\ref{eq03})  and it is  not
valid for problem (\ref{eq1})-(\ref{eq3}) in general case.
Additionally, if for a given \ $h\in [2^{-L},2^L]$ \ system
(\ref{sure1}) has a basic solution
$x_1^0,x_2^0,\dots,x_{n+q}^0,v_1^0,v_2^0,\dots,v_{p+1}^0$ with the
properties  mentioned in Theorem \ref{thm101} and Corollary
\ref{blin0}, then  such a basic solution may be not unique.
\end{remark}
Takin into account the presented above results we can prove the
following theorem.
\begin{theorem}\label{cor102}
 Let disjoint bilinear programming problem (\ref{eq01})-(\ref{eq03}) be such that  system (\ref{eq03}) satisfies conditions
$a)-c)$ and the set of solutions $X$ of system (\ref{eq2}) is
nonempty and bounded.  Then  the minimal value $z^*$ of the
objective function of problem (\ref{eq01})-(\ref{eq03}) is equal to
the minimal value \ $h^*$ \ of parameter \ $h$ \ in system
(\ref{sure1}) for which this system has a basic feasible solution
$x_1^*,\ x_2^*,\ \dots,\ x_{n+q}^*,v_1^*,\ v_2^*,\ \dots,v_{p+1}^*$
 where either the system of column vectors
\begin{equation}\label{eq261}
\bigg\{D_k=\left(\begin{array}{c}
 d_{k 1}\\[1.5mm]
 d_{k 2}\\[1.5mm]
  \vdots\\
 d_{k m}
 \end{array}
 \right):\ v_k^*>0,  \quad k\in \{=1,2,\dots,p\}\bigg\}
\end{equation} is linearly independent or this system of vectors is an empty set.  If such a solution for system (\ref{sure1}) with $h=h^*$ is
known, then $x_1^*,x_2^*,\dots, x_n^*$ together with an arbitrary
solution  $y_1^*,y_2^*,\dots,y_m^*$ of the  system
\begin{equation}\label{ghdm1}
\left\{%
\begin{array}{ll}
\displaystyle \quad \sum_{j=1}^m d_{kj} y_j \leq d_{k0}, \ \ \ \ \ \ k=1,2,\dots,p,\  \\
\displaystyle \quad \sum_{j=1}^m (\sum_{k=1}^p d_{kj} v_k^*) y_j =
\sum_{k=1}^p d_{k0} v_k^*
\end{array}%
\right.
\end{equation}
represents  an optimal solution  $x_1^*,x_2^*,\dots,x_{n+q}^*,
y_1^*,y_2^*,\dots,y_m^*$ of disjoint programming problem
(\ref{eq01})-(\ref{eq03}).  If  the system of column vectors
(\ref{eq261}) is an empty set, then  $v_1^*=0, v_2^*=0,\dots,
v_{p+1}^*=0$ and  $\sum_{k=1}^p d_{kj} v_k^* =0,$ \ $j =1,2,\dots,
m$; in this case $x_1^*,x_2^*,\dots,x_q^*$ \ together with an
arbitrary solution $y_1,y_2,\dots, y_m$ of system (\ref{eq03})
represents a solution of problem (\ref{eq01})-(\ref{eq03}).
\end{theorem}
\begin{proof}   Let $h^*$ be the minimal value of parameter $h$ for which system (\ref{sure1}) has a basic feasible solution $x_1^*,\ x_2^*,\ \dots,\
x_{n+q}^*,v_1^*,\ v_2^*,\ \dots,\ v_{p+1}^*$ for which either  the
system of vectors (\ref{eq261}) is linearly independent or this
system of vectors is an empty set. Then  according to   Theorem
\ref{thm101} and Corollary \ref{corr00} of Theorem \ref {tm01}  the
optimal value of the objective function of problem
(\ref{eq01})-(\ref{eq03}) with properties $a)-c)$ is equal to $h^*$
and $v_{p+1}^*=0$.

Now let us prove the second part of the theorem. According to
Corollary \ref{minmax11} of Theorem \ref{extrthm2},  for disjoint
bilinear programming problem (\ref{eq01})-(\ref{eq03}) we can
consider the following min-max problem:\\
\emph{Find}
$$h^*=\min_{(x_1,x_2,\dots,x_n)\in
X}\ \ \max_{(v_1,v_2,\dots,v_p)\ \in V(x_1,x_2,\dots,x_n)}\ \ \ \
\bigg(\sum_{j=1}^n g_j x_j\ -\sum_{k=1}^p d_{k0}v_k\bigg)$$
\emph{and} \ \ $(x_1^*,\ x_2^*,\ \dots,\ x_{n}^*)\in X,$
$$X=\!\bigg\{\!\!(x_1,x_2,\dots,x_n)\big|\sum_{i=1}^na_{ij}x_j\leq\! a_{i0}, \  i=1,2,\dots,q; \ \ x_i\geq \!0, \
i=1,2,\dots,q\bigg\},$$ \emph{such that} \ \
$$h^*=\max_{(v_1,v_2\dots,v_p)\ \in
V(x_1^*,x_2^*,\dots,x_n^*)}\ \ \ \bigg(\sum_{i=1}^n g_j x_j^*\
-\sum_{k=1}^p d_{k0}v_k\bigg),$$ \emph{where}
$$V(x_1,x_2,\dots,x_n)\!=\!\bigg\{\!\!(v_1,v_2,\dots,v_p)\big| -\sum_{k=1}^pd_{kj}\!v_k=\!\sum_{i=1}^nc_{ij}x_i+e_j,   j=\!1,2,\dots m\bigg\}.$$
We can observe that  $x_1^*,\ x_2^*,\ \dots,\ x_{n+q}^*,v_1^*,\
v_2^*,\ \dots,\ v_{p+1}^*$ \ is   a solution of this min-max problem
and \ $ x^*=(x_1^*,\ x_2^*,\ \dots,\ x_{n+q}^*) $ is an optimal
point for   problem (\ref{eq01})-(\ref{eq03}) with properties
$a)-c)$.
 Taking into account that $\sum\limits_{k=1}^pd_{kj}
v_k^* =-\sum\limits_{i=1}^nc_{ij}x_i^*-e_j, \ j=1,2,\dots,m$, and
$\sum\limits_{k=1}^p d_{k0} v_k^*=- h^*+\sum\limits_{i=1}^pg_ix_i^*$
we obtain that  system (\ref{ghdm1}) coincides with system
(\ref{eq026}), because $s_j^*=\sum\limits_{k=1}^pd_{kj} v_k^*$ and
$s_0^*=\sum\limits_{k=1}^pd_{k0}v_k^*$. So, for the optimal  point $
x^*=(x_1^*,\ x_2^*,\ \dots,\ x_{n+q}^*) $  the corresponding optimal
point $y^*=(y_1^*, y_2^*,\dots, y_m^*)$ for  problem
(\ref{eq01})-(\ref{eq03}) can be found by solving system
(\ref{ghdm1}). If  the system of column vectors (\ref{eq261}) is an
empty set, then it is evident that  $v_1^*=0, v_2^*=0,\dots,
v_{p+1}^*=0$ and $\sum_{k=1}^p d_{kj} v_k^* =0,$ \ $j =1,2,\dots,
m$; in this case $x_1^*,x_2^*,\dots,x_q^*$ \ together with an
arbitrary solution $y_1,y_2,\dots, y_m$ of system (\ref{eq03})
represents a solution of problem (\ref{eq01})-(\ref{eq03}), i.e.
have the case of Corollary \ref{corr00}.
\end{proof}
\section{\!The main results concerned with  checking the conditions of Theorem \ref{thm101}}\label{polalg} The main
results of this section are concerned with studying and solving the
following problem: to determine  if for a given $h$ system
(\ref{sure1})  has a basic feasible solution
$x_1^0,x_2^0,\!\dots,x_{n+q}^0,\! v_1^0,v_2^0,\dots,v^0_{p+1},$
where $v_{p+1}^0$ is a basic component. We show that for the
considered problem there exists a polynomial algorithm that
determines if such a basic solution for system (\ref{sure1}) exists
or doesn't exist. According to Corollary \ref{blin0} of Theorem
\ref{thm101}, a basic solution for system (\ref{sure1}) that has
$v_{p+1}^0$ as a basic component exists if and only if $Y \not
\subset Y_h$. We present the results that allow to check if the
condition $Y \not \subset Y_h$ holds and show how to determine a
basic solution $x_1^0,x_2^0,\!\dots,x_{n+q}^0,\!
v_1^0,v_2^0,\dots,v^0_{p+1}$ for
 system (\ref{sure1}) where $v_{p+1}^0$ is a basic component
when the condition $Y \not \subset Y_h$  holds. Note that all
results presented in this section are related to the case when
system (\ref{eq1710}) has no solutions, because in the case when
this system is consistent the solution of the problem can be easy
found.
\subsection{Some preliminary results for the problem of checking  if $Y \not \subset
Y_h$}\label{ldmd1} Let us consider system (\ref{sure1}) in the
following matrix form
\begin{equation}\label{reneq1}
\left\{%
\begin{array}{ll}
W\alpha=W_0, \\
\ \  \ \ \alpha\geq 0,
\end{array}%
\right.
\end{equation}
where $\alpha$ is the column vector of variables of system
(\ref{sure1}), i.e.
$$\alpha^T=(x_1,x_2,\dots x_n, x_{n+1}\dots, x_{n+q},
v_1,v_2,\dots,v_{p+1}),$$  $W=(w_{i,j})$ \ is  the \
$(q+m+1)\times(n+q+p+1)$-matrix of coefficients of system
(\ref{sure1}), and \ $W_0$ \ is the column vector of right hand
members of this system, \ i.e.
\begin{equation}\label{reneq2}
\arraycolsep 5pt W =\left ( \begin{array}{llll}
A           &      \      I     &    \  \ \  0       &   \    0\\[1.2mm]
C^T           &      \      0     &    \  \    D^T       &   \    0 \\[1.2mm]
g          &      \      0     &           -d^T   &   \    1
\end{array}\right), \arraycolsep 2pt \end{equation}
and
$$W_0^T=(w_{1,0},w_{2,0},\!\dots,w_{q+m+1,0})=(a_{1,0},a_{2,0},\!
\dots,a_{q,0},-e_1,-e_2,\dots,-e_m,h).$$ In the following, we will
denote the components of column vector \ $\alpha$ \  by \
$\alpha_j$, i.e.
$$\alpha^T=(\alpha_1,\alpha_2,\dots,\alpha_{n+q},\alpha_{n+q+1},\dots,
\alpha_{n+q+p+1}),$$ and will take into account that \ \ $\alpha_1,
\ \alpha_2, \ \dots, \ \alpha_{n+q+p+1}$ \ \ represent the
corresponding variables \ $x_1,\ x_2,\dots, \ x_{n+q}, \ v_1,\
v_2,\dots, v_{p+1}$  \ for system  \ (\ref{sure1}). We denote the
column vectors of matrix \ $W$ \ by  $W_1,\ W_2,\ \dots,\
W_{n+q+p+1}$, where
$$W_j^T=(w_{1,j}, \ w_{2,j}, \ \dots,\ w_{q+m+1,j}),\ \ \ j=1, \
2,\ \dots,\ n+q+p+1 \ \ (m\leq p)$$ and each $W_j$ corresponds to
the variable $\alpha_j$.

 Recall that
$$C=(c_{ij})_{n\times m},\ \  A=(a_{ij})_{q\times
n},\ \ D=(d_{ij})_{p\times m}, \  a^T=(a_{10},  a_{20}, \dots
a_{q0}),$$
$$ d^T=(d_{10},d_{20},\dots, d_{m0}), \ g= (g_1, \ g_2, \dots, \ g_n), \
e=(e_1,\ e_2, \ \dots, \ e_m),$$ $ \ \ \ \ \ rank(A)=q,\ \ \ \ \ \
rank(D)=m \ \ \  $ \ \ \ \ and \ \ \ \ \ \ $ rank(W)= q+m+1.$

\medskip Using the vector notations \  $x^T=(x_1, \ x_2, \
\dots, \ x_n),  \ \ \overline{x}^T=(x_{n+1}, \ x_{n+2}, \ \dots, \
x_{n+q}),$ \ $ v^T=(v_1, \ v_2, \ \dots, \ v_p)$, system
(\ref{reneq1}) can be written as follows
\begin{equation}\label{reneq3}
\left\{%
\begin{array}{ll}
Ax \ +\  I \overline{x} \quad\quad\quad\quad\quad\quad \ \ \ \ \ \!= \ \ a^T,\\
\!C^T\!x \  \quad\quad\quad-  D^Tv   \quad\quad\quad \ \  =-e^T,\\
\  gx  \ \quad\quad\quad- \ \!d^T\!v \ \ + \ v_{p+1}=\ \ h,\\
\  \ x\geq\!0, \  \overline{x}\geq0, \ v\geq 0, \  \!v_{p+1}\geq \ \
0.
\end{array}%
\right.
\end{equation}

Let us assume that for a given $h\in [2^L, 2^L]$, system
(\ref{reneq1}) has  a basic feasible solution ${\alpha^0}^T$.
Without loss of generality, we may assume that the basic components
of ${\alpha^0}^T\!\!$ are
$\alpha_1^0,\alpha_2^0,\dots,\alpha_{q+\!m\!+\!1}^0$ which
correspond to the linearly independent column vectors
$W_1,W_2,\dots,W_{q+m+1}$ that form a basis for the set of column
vectors of matrix $W$  (if it is not so, then we can relabel
variables such that the first $q+m+1$ components of  ${\alpha^0}^T$
are the basic ones). We denote by $W_B$ the matrix consisting of
column vectors $W_1,W_2,\dots,W_{q+m+1}$, and by $W_N$ the matrix
consisting of column vectors $W_{q+m+2},\ W_{q+m+3},\ \dots,\
W_{n+q+p+1}$. Then system (\ref{reneq1}) can be represented as
$$\left\{%
\begin{array}{ll}
W_B \cdot\alpha_B\ +\ W_N \cdot\alpha_N=W_0, \\[1.5mm]
 \ \  \quad \  \ \alpha_B\geq 0, \quad\quad \alpha_N\geq 0,
\end{array}%
\right.$$ where
$\alpha_B^T=(\alpha_1,\alpha_2,\dots,\alpha_{q+m+1}), \ \
\alpha_N^T=(\alpha_{q+m+2},\alpha_{q+m+3},\dots,\alpha_{n+q+p+1}).$
For $W_B$ there exists $W_B^{-1}$, therefore this system is
equivalent to the system
$$\left\{%
\begin{array}{ll}
I\cdot\alpha_B+W_B^{-1}\cdot W_N\cdot\alpha_N=W_B^{-1}\cdot
W_0,\\[1.5mm]
 \ \  \ \   \alpha_B\geq 0, \quad\quad\quad\ \ \ \alpha_N\geq 0.
\end{array}\\
 \right.$$
By introducing the notations ${\bf{W}}=W_B^{-1}\cdot W_N$ and \ $
{\bf{W}}_0=W_B^{-1}\cdot W_0$, we obtain the system
\begin{equation}\label{qeq01}\left\{%
\begin{array}{ll}
\alpha_B+{\bf{W}}\cdot\alpha_N={\bf{W}}_0,\\[1.5mm]
 \alpha_B\geq 0, \quad \alpha_N\geq 0,
\end{array}\\
 \right.
 \end{equation}
where \ \  ${\bf{W}}=({\bf{w}}_{i,j})$ \ is \ a \ $(q+m+1)\times
(n+p-m)$-matrix \ \  and ${\bf{W}}_0$ is a column vector with
nonnegative components
${\bf{w}}_{1,0},{\bf{w}}_{2,0},\dots,{\bf{w}}_{q+m+1,0}$, \ i.e.\
$\alpha_B^0={\bf{W}}_0.$

Now let us consider system (\ref{qeq01}) in the following unfolded
form:
\begin{equation}\label{blin2}
\left\{%
\begin{array}{l}
\alpha_1  \ \  \quad\quad\quad\ \
+\sum\limits_{j=r+1}^{r+s-1}{\bf{w}}_{1,j}\alpha_j+{\bf{w}}_{1,r+s}\alpha_{r+s}={\bf{w}}_{1,0},\\[0.3mm]
\ \ . \ \ . \ \ \ . \ \ .\ \ . \ \ . \ \ . \ \ . \ \ . \ \ . \ \ . \
\ . \ \ . \ \ . \ \ . \ \ . \ \ . \ \ . \ \ . \ \ . \ \ . \ \ . \ \
.
\ \ \\[0.3mm]
\quad\quad\quad\alpha_{k}\quad  \ +\sum\limits_{j=r+1}^{r+s-1}{\bf{w}}_{k,j}\alpha_j\!+{\bf{w}}_{k,r+s}\alpha_{r+s}\!={\bf{w}}_{k,0},\\[0.3mm]
\ \ . \ \ . \ \ \ . \ \ .\ \ . \ \ . \ \ . \ \ . \ \ . \ \ . \ \ . \
\ . \ \ . \ \ . \ \ . \ \ . \ \ . \ \ . \ \ . \ \ . \ \ . \ \ . \ \
.
\ \ \\[1mm]
\quad \quad\quad\quad\ \  \alpha_r+\sum\limits_{j=r+1}^{r+s-1}{\bf{w}}_{r,j}\alpha_j+{\bf{w}}_{r,r+s}\alpha_{r+s}={\bf{w}}_{r,0},\\[0.5mm]
\qquad\quad\quad\quad\quad\quad\quad \alpha_j\geq 0, \quad j=1,\ 2,\
\dots,\ r+s,
\end{array}%
\right.
\end{equation}
where \  $r=q+m+1,\  \  s=n+p-m$, \ and \   ${\bf{w}}_{i,0}\geq 0, \
 i=1,\ 2, \ \dots,\ r$. \
So, ${\alpha^0}^T=(\alpha_1^0, \alpha_2^0,  \dots, \alpha_r^0,  0,
0,\dots, 0)$ is a basic feasible solution of system (\ref{blin2}),
where $\alpha_i^0={\bf{w}}_{i,0}, \ i=1,2,\dots,r$. Obviously, if in
(\ref{blin2}) one of the basic component $\alpha_i, \
i\in\{1,2,\dots,r\}$ corresponds to $v_{p+1}$ then
$v_{p+1}^0=\alpha_i^0$ is the basic component for ${\alpha^0}^T$ and
according to Theorem \ref{blin0} we have  $Y \not \subset Y_h$.
 If in (\ref{blin2}) among $\alpha_1,\alpha_2,\dots,\alpha_r$  there is
no  component $\alpha_i, \ i\in\{1,2,\dots,r\}$ that corresponds to
$v_{p+1}$, then without loss of generality we may assume that
$v_{p+1}$ in this system is represented by the nonbasic variable
$\alpha_{r+s}$ with the corresponding column vector
${\bf{W}}_{r+s}$,
${\bf{W}}_{r+s}^T=({\bf{w}}_{1,r+s},{\bf{w}}_{2,r+s},\dots,{\bf{w}}_{r,r+s}),$
where $\alpha_{r+s}=v_{p+1}$ and at least a component
${\bf{w}}_{i,r+s}$ of ${\bf{W}}_{r+s}$ is different from zero. If it
is not so, we always can relabel variable such that $\alpha_{r+s}$
will correspond to $v_{p+1}$. Then in this case the following lemma
holds.
\begin{lemma}\label{izar1}
Let  ${\alpha^0}^T=(\alpha_1^0, \alpha_2^0, \ \dots,\ \alpha_r^0, \
0, \ 0,\dots,\ 0)$  be  a basic feasible solution of system
(\ref{blin2}) where the nonbasic variable $\alpha_{r+s}$ in
(\ref{blin2}) corresponds to $v_{p+1}$. If the column vector
${\bf{W}}_{r+s}$ \ contains positive components, then $Y \not
\subset Y_h$. If the column vector ${\bf{W}}_{r+s}$ contains a
nonzero component ${\bf{w}}_{i_0,r+s}$ (positive or negative)  and \
${\bf{w}}_{i_0,0}=0$,  then $Y \not \subset Y_h$.
\end{lemma}
\begin{proof}  Assume that column vector ${\bf{W}}_{r+s}$ contains positive components. Then
we can find $i_0\in\{1,2,\dots,r\}$ such that
$$\frac{{\bf{w}}_{i_0,0}}{{\bf{w}}_{i_0,r+s}}=\min_{{\bf{w}}_{i,r+s}>0}\frac{{\bf{w}}_{i,0}}{{\bf{w}}_{i,r+s}}$$
and  determine a new basic feasible solution $\alpha'$ by applying
for system (\ref{blin2}) the standard pivoting procedure with pivot
element ${\bf{w}}_{i_0,r+s}$ in the same way as in the simplex
algorithm for linear programming. After such a pivoting procedure we
determine a new basic feasible solution $\alpha'$, the components of
which  are calculated as follows:
$$\alpha'_i=\left\{%
\begin{array}{ll}
{\bf{w}}_{i,0}-\displaystyle\frac{{\bf{w}}_{i,r+s}
{\bf{w}}_{i_0,0}}{{\bf{w}}_{i_0,r+s}}\ \ \ \
\hbox{for} \ \ \ i=1,2,\dots,r(i\neq i_0, r+s),\\[3.0mm]
\ \ \ 0\ \qquad \qquad\qquad\quad \ \ \hbox{for}\quad i= i_0, r+1, r+2,\dots,r+s-1,\\[2.0mm]
\displaystyle\frac{{\bf{w}}_{i_0,0}}{{\bf{w}}_{i_0,r+s}}\
\quad\quad\quad\quad \quad \ \  \hbox{for}\ \ \ i=r+s.
\end{array}%
\right.$$  In such a way we  determine the basic feasible solution
$${\alpha'}^T=(\alpha'_1,\ \alpha'_2,\ \dots,\ \alpha'_{i_0-1}, \ \alpha'_{r+s},\ \alpha'_{i_0+1},\ \dots,\ \alpha'_r,\ 0, \ 0,\ \dots,\ 0, \ 0 )$$
that in initial system (\ref{reneq1}) corresponds to the basis
$$\{W_1,W_2,\dots,W_{i_0-1},W_{r+s},W_{i_0+1},\dots,W_r\}$$ obtained
from $$\{W_1,W_2,\dots,W_{i_0-1},W_{i_0},W_{i_0+1},\dots,W_r\}$$ by
replacing the column vector $W_{i_0}$ with the column vector
$W_{r+s}$. We can observe that the basic component $\alpha'_{r+s}$
of this basic feasible solution corresponds to $v_{p+1}$ and
therefore according to Theorem \ref{blin0} we have $Y \not \subset
Y_h$.

In the case ${\bf{w}}_{i_0,0}=0$,  ${\bf{w}}_{i_0,r+s}\neq 0$ where
${\bf{w}}_{i_0,r+s}$ may be positive or negative, we also can apply
for system (\ref{blin2}) the same pivoting procedure with the pivot
element ${\bf{w}}_{i_0,r+s}$ in spite of the fact that the sign of
${\bf{w}}_{i_0,r+s} $  may be negative. After such a procedure we
determine the basic feasible solution
$$
{\alpha'}^T=(\alpha'_1,\alpha'_2,\dots, \alpha'_{i_0-1},
\alpha'_{r+s}, \alpha'_{i_0+1},\dots,\alpha_r, 0, 0,\dots, 0 ),$$
for which $\alpha'_{r+s}=\alpha^0_{i_0}=0$ and \ $\alpha^0=\alpha'$,
i.e. $\alpha'$ is the basic solution that corresponds to the basis
$W_1,W_2,\dots,W_{i_0-1},W_{r+s},W_{i_0+1},\dots,W_r$, obtained from
 $W_1,W_2,\dots,W_{i_0-1},W_{i_0},W_{i_0+1},\dots,W_r$, by
 replacing
$W_{i_0}$ with $W_{r+s}$. So, $\alpha'$ is a basic feasible solution
for which the component $\alpha'_{r+s}$ corresponds to $v_{p+1}$,
and therefore according to Theorem \ref{blin0} we have $Y \not
\subset Y_h$.
\end{proof}
\begin{corollary}\label{izar2}
If  $Y \subset Y_h$, then  for an arbitrary  basic feasible solution
${\alpha^0}$ of system (\ref{blin2}) the  column vectors
${\bf{W}}_{r+s}$ possess the following properties:

1) all components of  ${\bf{W}}_{r+s}$ are non-positive;

2) if ${\bf{w}}_{i,r+s}<0 $ for  some $i\in \{1,2,\dots,r\}$, then
${\bf{w}}_{i,0}>0.$
\end{corollary}
\subsection{The main results for checking if $Y\not\subset Y_h$}\label{ldmd2} In the previous
subsection we showed that the problem of the existence of a basic
solution  for system (\ref{sure1}) that has $v_{p+1}$ as a basic
component can be reduced to the problem of the existence of a basic
solution   for system (\ref{blin2}) that has $\alpha_{r+s}$ as the
basic component, where $\alpha_{r+s}$ corresponds to $v_{p+1}$. In
this subsection, we show
how to determine if such a solution for system (\ref{blin2}) exists
in the case when the column vectors ${\bf {W}} _ {r + s} $ and $
{\bf {W}} _ 0 $ satisfy the following condition
\begin{equation}\label{compl1}
 \left\{%
\begin{array}{l}
{\bf{w}}_{i,r+s}\leq 0, \ i=1,2,\dots,r;\qquad\qquad\qquad\qquad\qquad\qquad\qquad\qquad\\[2mm]
 if \ {\bf{w}}_{i,r+s}<0 \ for \ some \ i\in \{1,2,\dots,r\}, \  then \ {\bf{w}}_{i,0}>0.\qquad\qquad
\end{array}%
\right.
\end{equation}
If this condition is not satisfied, then according to Lemma
\ref{izar1} we have $Y \not \subset Y_h$. If condition
(\ref{compl1}) for system (\ref{blin2}) is satisfied, then by
relabeling its equations and variable we can represent this system
in the following form
\begin{equation}\label{blinn21}
\left\{%
\begin{array}{l}
\alpha_1 \ \ \  \quad\quad\quad \ \
+\sum\limits_{j=r+1}^{r+s-1}{\bf{w}}_{1,j}\alpha_j \quad\quad\quad\quad\quad \ \ \ \ \ \!={\bf{w}}_{1,0},\\[0.3mm]
\ \ . \ \ . \ \ \ . \ \ .\ \ . \ \ . \ \ . \ \ . \ \ . \ \ . \ \ . \
\ . \ \ . \ \ . \ \ . \ \ . \ \ . \ \ . \ \ . \ \ . \ \ . \ \ . \ \ . \ \    \\[0.3mm]
\quad\ \ \alpha_{k-1}\ \ \ \  \ +\sum\limits_{j=r+1}^{r+s-1}\!{\bf{w}}_{k-\!1,j}\alpha_j \ \quad\quad\quad\quad\quad \ \ ={\bf{w}}_{k-\!1,0},\\[1mm]
\quad\quad\quad\  \alpha_{k}\  \ \ +\sum\limits_{j=r+1}^{r+s-1}{\bf{w}}_{k,j}\alpha_j\ + \ {\bf{w}}_{k,r+s}\alpha_{r+s}\!={\bf{w}}_{k,0},\\[0.3mm]
\ \ . \ \ . \ \ . \ \ . \ \ . \ \ . \ \  .\ \ . \ \ . \ \ . \ \ . \
\ . \ \ . \ \ . \ \ . \ \ \ . \ \ .\ \ . \ \ . \ \ . \ \ . \ \ . \ \
.
\ \   \\[0.3mm]
\quad \quad\quad\quad\ \ \ \alpha_r\!+\sum\limits_{j=r+1}^{r+s-1}{\bf{w}}_{r,j}\alpha_j \ + \ {\bf{w}}_{r,r+s}\alpha_{r+s}\!={\bf{w}}_{r,0},\\[0.5mm]
\qquad\quad\quad\quad\quad\quad\quad \alpha_j\geq 0, \quad j=1,\ 2,\
\dots,\ r+s,
\end{array}%
\right.
\end{equation}
where  \ ${\bf{w}}_{i,0}\geq 0, \ i=1, 2, \dots, k-1;$ \
${\bf{w}}_{i,0}>0, \ i=k,\ k+1,\ \dots,\ r$,\ and \
${\bf{w}}_{i,r+s}\!<0,$ \ $i=k,k\!+\!1,\dots,r$. In this system, we
will assume  that the variable $\alpha_{r+s}$ corresponds to the
variable $v_{p+1}$ of system (\ref{sure1}). Note that  in
(\ref{blinn21}) the variables \
$\alpha_k,\alpha_{k+1},\dots,\alpha_r$ correspond to some of the
variables \ $v_1,v_2,\dots,v_p$ of system (\ref{sure1}), and if a
variable $\alpha_i$ corresponds to one of the bounded variables
$x_1, x_2,\dots, x_{n+q}$ of system (\ref{sure1}), then $1\leq i\leq
k$, where $q+n+1\leq k\leq r$.

\medskip
 It can be easily observed that if for an
arbitrary feasible solution of system (\ref{blinn21}) the following
condition holds
\begin{equation}\label{landd1}\sum\limits_{j=r+1}^{r+s-1}{\bf{w}}_{r,j}\alpha_j
<{\bf{w}}_{r,0},
\end{equation}
then  the problem of the existence of a basic feasible solution  for
system (\ref{blinn21}) that has $\alpha_{r+s}$ as a basic component
can be reduced to a similar problem for a new system with fewer
equations and fewer variables. Such a reduction can be made  on the
bases of the following lemma.
\begin{lemma}\label{check1} Let the  consistent system (\ref{blinn21}) with  $k<r$ be given. If for an arbitrary feasible solution of this system
condition (\ref{landd1}) holds, then system (\ref{blinn21}) has a
basic feasible solution ${\alpha^0}^T=(\alpha^0_1,\alpha^0_2,\dots,
\alpha^0_{r+s})$ that  $\alpha^0_{r+s}$ is a basic component if and
only if the following system
\begin{equation}\label{check01}
\left\{%
\begin{array}{l}
\alpha_1 \ \ \  \quad\quad\quad\ \ \ \ \
+\sum\limits_{j=r+1}^{r+s-1}{\bf{w}}_{1,j}\alpha_j \quad\quad\quad\quad\quad\ \ \ \ \ \ \ \ ={\bf{w}}_{1,0},\\[0.3mm]
\ \ . \ \ . \ \ \ . \ \ .\ \ . \ \ . \ \ . \ \ . \ \ . \ \ . \ \ . \
\ . \ \ . \ \ . \ \ . \ \ . \ \ . \ \ . \ \ . \ \ . \ \ . \ \ . \ \
.
\ \    \\[2mm]
\quad\quad\ \alpha_{k-1}\  \  \ \ \ \ \  +\sum\limits_{j=r+1}^{r+s-1}\!{\bf{w}}_{k-\!1,j}\alpha_j\ \ \ \ \quad\quad\quad\quad\quad \ \ ={\bf{w}}_{k-\!1,0},\\[1mm]
\quad\quad\quad\ \ \alpha_{k} \ \ \ \ \ \ +\sum\limits_{j=r+1}^{r+s-1}{\bf{w}}_{k,j}\alpha_j\ \ \ + \ \ \ {\bf{w}}_{k,r+s}\alpha_{r+s}\!={\bf{w}}_{k,0},\\[0.3mm]
\ \ . \ \ . \ \ \ . \ \ .\ \ . \ \ . \ \ . \ \ . \ \ . \ \ . \ \ . \
\ . \ \ . \ \ . \ \ . \ \ . \ \ . \ \ . \ \ . \ \ . \ \ . \ \ . \ \
.
\ \    \\[0.3mm]
\quad \quad\quad\quad\ \ \ \alpha_{r-1}  +\!\sum\limits_{j=r+1}^{r+s-1}{\bf{w}}_{r-1,j}\alpha_j \ +  {\bf{w}}_{r-1,r+s}\alpha_{r+s}\!={\bf{w}}_{r-1,0},\\[0.3mm]
\qquad\quad\quad\quad\quad\quad\quad \quad\quad \alpha_j\geq 0,\
\forall j\in\{1,2,\dots,r+s\}\setminus\{r\},
\end{array}%
\right.
\end{equation}
has a basic feasible solution \
${\overline{\alpha}^0}^T=(\alpha^0_1,\alpha^0_2,\dots,\alpha^0_{r-1},\alpha^0_{r+1}\dots\alpha^0_{r+s})$
\ that  \ $\alpha^0_{r+s}$ is a basic component. If \ $k=r$ \ and
for an arbitrary feasible solution of system \ (\ref{blinn21}) \
condition  \ (\ref{landd1}) \ holds, then system \ (\ref{blinn21}) \
has no  basic feasible solution
${\alpha^0}^T=(\alpha^0_1,\alpha^0_2,\dots, \alpha^0_{r+s})$ that
$\alpha^0_{r+s}$ is a basic component.
\end{lemma}
\begin{proof} Indeed, if $k<r$ and  for an arbitrary
solution of system (\ref{blinn21}) condition (\ref{landd1}) holds,
then $\alpha_r>0$ for an arbitrary solution of system
(\ref{blinn21}), i.e. $\alpha_r$ is the basic component for an
arbitrary basic solution of this system. Therefore we can eliminate
the last equation from (\ref{blinn21}) and conclude that system
(\ref{blinn21}) has a basic solution
${\alpha^0}^T=(\alpha^0_1,\alpha^0_2,\dots,\alpha^0_{r+s})$ such
that $\alpha^0_{r+s}$ is a basic  component if and only if system
(\ref{check01}) has a basic solution
${\overline{\alpha}^0}^T=(\alpha^0_1,\alpha^0_2,\dots,\alpha^0_{r-1},\alpha^0_{r+1}\dots\alpha^0_{r+s})$
such that $\alpha^0_{r+s}$ is  a basic  component. In the case $k=r$
it is evident that system (\ref{blinn21}) couldn't contain  a basic
feasible solution ${\alpha^0}^T=(\alpha^0_1,\alpha^0_2,\dots,
\alpha^0_{r+s})$ that has $\alpha^0_{r+s}$ as a basic component.
\end{proof}
\begin{corollary}\label{clly}
If for an arbitrary feasible solution of system (\ref{blinn21}) the
following  condition holds
$$\sum\limits_{j=r+1}^{r+s-1}{\bf{w}}_{i,j}\alpha_j <
{\bf{w}}_{i,0}, \ i=k,k+1,\dots,r,$$ then system (\ref{blinn21}) has
no  basic feasible solution
${\alpha^0}^T=(\alpha^0_1,\alpha^0_2,\dots,\alpha^0_{r+s})$ that has
$\alpha^0_{r+s}$ as a basic component.
\end{corollary}

  In the case when for an arbitrary feasible solution  of  system  (\ref{blinn21}) the following condition holds
\begin{equation}\label{land1}\sum\limits_{j=r+1}^{r+s-1}{\bf{w}}_{r,j}\alpha_j
\leq {\bf{w}}_{r,0}
\end{equation}
 Lemma \ref{check1} can be specified  as
 follows.
\begin{lemma}\label{check3} Assume that for an arbitrary feasible
solution of consistent system (\ref{blinn21}) with \ $k\leq r$ \
condition (\ref{land1}) holds, and at least for a  feasible
solution ${\alpha}^T=(\alpha_1,\alpha_2,\dots,\alpha_{r+s})$  of
system (\ref{blinn21}) condition (\ref{land1})  is satisfied as
equality, i.e.
$\sum\limits_{j=r+1}^{r+s-1}{\bf{w}}_{r,j}\alpha_j={\bf{w}}_{r,0}.$
Then  system (\ref{blinn21})  has a basic feasible solution
 ${\alpha^0}^T=(\alpha^0_1,\alpha^0_2,\dots,\alpha^0_{r+s})$
that has  $\alpha^0_{r+s}=0$ as a degenerate basic component. Such a
basic solution for system (\ref{blinn21})   can be found by using
 an  optimal  basic  solution
${\overline{\alpha}^0}^T=(\overline{\alpha}^0_1,\overline{\alpha}^0_2,\dots,\overline{\alpha}^0_{k-1},\overline{\alpha}^0_{r+1},
\overline{\alpha}^0_{r+2},\dots,\overline{\alpha}^0_{r+s-1})$ \ of
the following linear programming problem:\\
Maximize
\begin{equation}\label{doga1}z_r=\sum\limits_{j=r+1}^{r+s-1}{\bf{w}}_{r,j}\alpha_j
-{\bf{w}}_{r,0} \ \ \end{equation} subject to
\begin{equation}\label{check011}
\left\{%
\begin{array}{l}
\alpha_1  \ \  \quad\ \ \ \ \ \ \
+ \ \sum\limits_{j=r+1}^{r+s-1}{\bf{w}}_{1,j}\alpha_j \quad\quad\quad\ \      ={\bf{w}}_{1,0},\\[0.5mm]
\quad\ \alpha_2   \  \ \ \ \ \ \ \ + \ \sum\limits_{j=r+1}^{r+s-1}\!{\bf{w}}_{2,j}\alpha_j \quad\quad\quad\quad    ={\bf{w}}_{2,0},\\[0.3mm]
\ \ . \ \ . \ \ \ . \ \ .\ \ . \ \ . \ \ . \ \ . \ \ . \ \ . \ \ . \
\ . \ \ . \ \ . \ \ . \ \ . \ \ . \ \ . \ \ .     \\[0.3mm]
\quad\quad\  \alpha_{k-1} \  \ \ + \ \sum\limits_{j=r+1}^{r+s-1}\!{\bf{w}}_{k-\!1,j}\alpha_j \!\quad\quad\quad    ={\bf{w}}_{k-\!1,0},\\[0.3mm]
 \  \alpha_j\geq 0,\
 j=1,2,\dots,k\!-\!1,r\!+\!1,r\!+\!2,\dots,r\!+\!s-\!1.
\end{array}%
\right.
\end{equation}
If  \ ${\overline{\alpha}^0}^T$ \  is known, \ then a basic feasible
solution
 ${\alpha^0}^T=(\alpha^0_1,\alpha^0_2,\dots,\alpha^0_{r+s})$ of
system (\ref{blinn21}) that has  $\alpha^0_{r+s}=0$ as a degenerate
basic component can be found as follows
\begin{equation}\label{jora1}
\!\!\!\alpha^0_i\!=\!\left\{%
\begin{array}{lll}
\quad\quad\quad\quad    \overline{\alpha}^0_i, & \  \hbox{ if
\ $i=1,2,\dots,k\!-\!1,r\!+\!1,\dots,r\!+\!s\!\!-\!\!1$,}\\[1mm]
-\sum\limits_{j=r+1}^{r+s-1}{\bf{w}}_{i,j}\overline{\alpha}^0_j+{\bf{w}}_{i,0}, &  \  \hbox{ if \ $i=k,k+1,\dots,r-1$,}\\[1mm]
\quad\quad\quad\quad       0, & \ \
\hbox{if \  $i=r, \ r+s$,}\\
\end{array}
\right.
\end{equation}
where  $\alpha_r^0=0$ is  a nonbasic component of ${\alpha^0}^T$. In
this case inequality (\ref{land1}) is weekly redundant for the
following system
\begin{equation}\label{bingo28}
\left\{%
\begin{array}{l}
\sum\limits_{j=r+1}^{r+s-1}{\bf{w}}_{i,j}\alpha_j\leq{\bf{w}}_{i,0}, \quad\quad \ \!i=1,2,\dots,k\!-\!1,\\[0.5mm]
\quad\quad\quad\quad \  \alpha_j\geq 0,\quad\quad\quad\
j=r\!+\!1,r\!+\!2,\dots,r\!+\!s\!-\!1,
\end{array}%
\right.
\end{equation}
and an arbitrary  inequality
\begin{equation}\label{jora3}\sum_{j=r+1}^{r+s-1}{\bf{w}}_{i,j}\alpha_j\leq{\bf{w}}_{i,0},
\ i\in\{k,k+1,\dots,r\}
\end{equation}
 is redundant for system (\ref{bingo28}).
 \end{lemma}
\begin{proof} Assume that the conditions of the lemma are satisfied and consider the following linear programming problem:
 \emph{Maximize  (\ref{doga1}) subject to
(\ref{blinn21})}. Then this problem has solutions and the optimal
value of the objective function (\ref{doga1}) is equal to zero. If
we dualize the linear programming problem
(\ref{blinn21}), (\ref{doga1}), then we obtain the following problem:\\
 \emph{Minimize}
\begin{equation}\label{nata1}z'=\sum_{i=1}^{k-1}{\bf{w}}_{i,0}\beta_i \ + \
\sum_{i=k}^{r}{\bf{w}}_{i,0}\beta_i-{\bf{w}}_{r,0}\end{equation}
\emph{subject to}
\begin{equation}\label{nata2}\left\{%
\begin{array}{l}
\ \ \sum\limits_{i=1}^{k-1}{\bf{w}}_{i,j}\beta_i\ \ + \ \sum\limits_{i=k}^{r}{\bf{w}}_{i,j}\beta_i\  \geq {\bf{w}}_{r,j}, \ \  j=\overline{r+1,r+s-1},\\[0.5mm]
\quad\quad \quad\quad\quad\ \ \ \ \ \sum\limits_{i=k}^{r}{\bf{w}}_{i,r+s}\beta_i\geq 0,\\[0.3mm]
\quad \quad \ \ \ \ \ \beta_i\geq 0, \ \ i=1,2,\dots,r.
\end{array}%
\right.\end{equation}
This problem has solutions and the optimal
value of the objective function is  equal to zero. Taking into
account that \ ${\bf{w}}_{i,r+s}<0, \ i=k,\  k+1, \ \dots, \ r$ \ we
obtain  \ $\beta_i=0, \ i=k, k+1,\dots, r$. Therefore this dual
problem can
be written as follows: \\
\emph{Minimize}
\begin{equation}\label{bostan1}
z'=\sum_{i=1}^{k-1}{\bf{w}}_{i,0}\beta_i -{\bf{w}}_{r,0}
\end{equation}
\emph{subject to}
\begin{equation}\label{bostan2}
\left\{%
\begin{array}{l}
 \ \ \sum\limits_{i=1}^{k-1}{\bf{w}}_{i,j}\beta_i\geq {\bf{w}}_{r,j}, \ \  j=\overline{r+1,r+s-1},\\[0.3mm]
\quad \quad \quad \ \ \ \beta_i\geq 0, \ \ j=1,2,\dots,r+s-1.
\end{array}%
\right. \end{equation} Problem (\ref{bostan1}), (\ref{bostan2}) can
be regarded as the dual problem for linear programming problem
(\ref{doga1}), (\ref{check011}) and for linear programming problem
(\ref{doga1}), (\ref{bingo28}), i.e. for these problems there exist
optimal solutions and the corresponding optimal values of the
objective functions are equal to zero. So, the linear programming
problem (\ref{doga1}),(\ref{bingo28}) has  an optimal solution that
is attained in an extreme point \
$(\overline{\alpha}^0_{r+1},\overline{\alpha}^0_{r+2},\dots,\overline{\alpha}^0_{r+s-1})$
\ of the set of solutions of system \ (\ref{bingo28}) \ and the
optimal value of the objective function \ (\ref{doga1}) \ is equal
to zero. This means  that  \
$\sum\limits_{j=r+1}^{r+s-1}{\bf{w}}_{r,j}\overline{\alpha}_j^0
-{\bf{w}}_{r,0}=0$, i.e. inequality  (\ref{land1}) is weekly
redundant for system (\ref{bingo28}). Moreover, inequality
(\ref{land1}) is  weakly redundant for the following system
\begin{equation}\label{capjor1}
\left\{%
\begin{array}{ll}\displaystyle\sum_{j=r+1}^{r+s-1}{\bf{w}}_{i,j}\alpha_j\qquad\qquad\qquad\leq{\bf{w}}_{i,0},
\ \ i=1,2,\dots,k-1,\\[1mm]
\displaystyle\sum_{j=r+1}^{r+s-1}{\bf{w}}_{i,j}\alpha_j+{\bf{w}}_{i,r+s}\alpha_{r+s}\leq{\bf{w}}_{i,0},\
\ i=k,k+1,\dots,r,\\[1mm]
\quad\quad \ \alpha_j\geq 0, j=r+1,r+2,\dots,r+s,
\end{array}%
\right.
\end{equation}
where \
$(\overline{\alpha}_{r+1}^0,\overline{\alpha}_{r+2}^0,\dots,\overline{\alpha}^0_{r+s-1},\overline{\alpha}^0_{r+s})$
\  with \ $\overline{\alpha}_{r+s}^0=0$  represents an extreme point
of the set of solutions of system (\ref{capjor1}) in which the
maximal value of the objective function   (\ref{doga1}) on the set
of solutions of this system  is attained, i.e.
$z_r^0=\sum\limits_{j=r+1}^{r+s-1}{\bf{w}}_{r,j}\overline{\alpha}_j^0
-{\bf{w}}_{r,0}=0$. Therefore $(\overline{\alpha}^0_1,
\overline{\alpha}^0_2,  \dots,\overline{\alpha}^0_{r+s-1},
\overline{\alpha}^0_{r+s})$  with
$\overline{\alpha}^0_i=-\sum\limits_{j=r+1}^{r+s-1}{\bf{w}}_{i,j}\overline{\alpha}_j^0
+{\bf{w}}_{r,0}\geq 0,$ \ $i=1,2,\dots,r;$ \
$\overline{\alpha}_{r+s}^0=0,$ represents an optimal basic solution
of linear programming problem: \emph{Maximize the objective function
(\ref{doga1})} subject to (\ref{blinn21}). Thus, we may conclude
that ${\alpha^0}^T=(\alpha^0_1,\alpha^0_2,\dots,\alpha^0_{r+s})$
with the components determined according to (\ref{jora1}) represents
a basic
 feasible solution for  system \ (\ref{blinn21}) that has
\ $\alpha^0_{r+s}=0$ \ as the degenerate basic component.
 Indeed, from (\ref{jora1}) we can see that
${\alpha^0}^T$ contains no more than $k-1$  basic components of
${\overline{\alpha}^0}^T,$ no more than $r-k$ components
$\alpha_i^0=-\sum\limits_{j=r+1}^{r+s-1}{\bf{w}}_{i,j}\overline{\alpha}_j^0
+{\bf{w}}_{r,0}\geq 0,$ \ $i=k,k+1,\dots,r-1$, and the degenerate
basic component $\alpha_{r+s}^0=0$, i.e. in the whole we obtain that
${\alpha^0}^T$ contains no more than $r$ basic components, including
the degenerate basic component $\alpha_{r+s}^0=0$. In general, if
the conditions of Lemma \ref{check3} are satisfied, then formula
(\ref{jora1}) is valid  for an arbitrary basic feasible solution
${\alpha^0}^T=(\alpha^0_1,\alpha^0_2,\dots,\alpha^0_{r+s})$ of
system (\ref{blinn21}) with $\alpha_{r+s}^0=0$, because
$\alpha_i^0=-\sum\limits_{j=r+1}^{r+s-1}{\bf{w}}_{i,j}{\alpha}_j^0
+{\bf{w}}_{r,0}\geq 0$,\ $i=k,k+1,\dots,r-1$. Therefore an arbitrary
inequality (\ref{jora3}) is redundant for system (\ref{bingo28}).
 \end{proof}
\begin{corollary} The consistent system (\ref{blinn21}) with $k=r$ has
a basic feasible solution \
${\alpha^0}^T=(\alpha^0_1,\alpha^0_2,\dots,\alpha^0_{r+s})$ \ that
has  $\alpha^0_{r+s}=0$ as a degenerate  basic component if and only
if inequality (\ref{land1}) is weakly redundant for the following
system of inequalities
\begin{equation}\label{bingo39}
\left\{%
\begin{array}{l}
\sum\limits_{j=r+1}^{r+s-1}{\bf{w}}_{i,j}\alpha_j\leq{\bf{w}}_{i,0}, \quad\quad \ \!i=1,2,\dots,r\!-\!1,\\[0.5mm]
\quad\quad\quad\quad \  \alpha_j\geq 0,\quad\quad\quad\
j=r\!+\!1,r\!+\!2,\dots,r\!+\!s\!-\!1.
\end{array}%
\right.
\end{equation}
\end{corollary}
\begin{remark}\label{dima83}Lemma \ref{check3} has been proven assuming that for an
arbitrary feasible solution of system (\ref{blinn21}) inequality
(\ref{land1}) holds, and at least for a feasible solution of this
system inequality (\ref{land1}) is satisfied as equality. In
general, such a condition in this lemma can be considered with
respect to an arbitrary inequality
\begin{equation}\label{land88}\sum\limits_{j=r+1}^{r+s-1}{\bf{w}}_{i_0,j}\alpha_j
\leq {\bf{w}}_{i_0,0}, \ \ \ i_0\in\{k,k+1,\dots,r\},
\end{equation}
and based on this similarly we can prove  the  existence of a basic
solution   ${\alpha^0}^T=(\alpha^0_1,\alpha^0_2,\dots,
\alpha^0_{r+s})$ \ for system (\ref{blinn21}) that has
$\alpha^0_{r+s}=0$ as a degenerate basic component.
\end{remark}

Taking into account  Remark \ref{dima83} to Lemma \ref{check3} we
obtain the following result.
\begin{theorem}\label{capmt1}
 The consistent system (\ref{blinn21}) with  $k\leq r$  has  a basic feasible
 solution \ ${\alpha^0}^T=(\alpha^0_1, \ \alpha^0_2, \ \dots,\
\alpha^0_{r+s})$ \ that has  $\alpha^0_{r+s}=0$ \ as the basic
component if and only if there exists $i_0\in \{k,k+1,\dots,r\}$ for
which the corresponding inequality (\ref{land88}) is weakly
redundant for system (\ref{bingo28}) and an arbitrary inequality
(\ref{jora3}) is redundant for system (\ref{bingo28}).
\end{theorem}

Based on the results above we can prove the following result.
\begin{theorem}\label{capmt11} Let the consistent system (\ref{blinn21}) with
$k\leq r$ be given and consider the linear programming problem:
Maximize the objective function (\ref{doga1}) on the set of
solutions of system (\ref{blinn21}). If this linear programming
problem has solutions and the maximal value of the objective
function  is $z_r^*=t_r^*$ where \ $0\leq t_r^*<\infty,$ \ then \
system \ (\ref{blinn21}) \ has \ a \ basic \ feasible \ solution \
${\alpha^*}^T=(\alpha^*_1,\alpha^*_2,\dots,\alpha^*_{r+s})$ \ that
has  $\alpha^*_{r+s}=-\displaystyle\frac{1}{{\bf{w}}_{r,r+s}}t_r^*$
as a basic component and $a^*_r=0$. Such a basic solution for system
(\ref{blinn21}) can be found by using an optimal basic solution
${\overline{\alpha}^*}^T=(\overline{\alpha}^*_1,\overline{\alpha}^*_2,\dots,\overline{\alpha}^*_{k-1},\overline{\alpha}^*_{r+1},
\overline{\alpha}^*_{r+2},\dots,\overline{\alpha}^*_{r+s-1})$ of the
linear programming problem: Maximize the objective function
(\ref{doga1}) on the set of solutions of system \ (\ref{check011}).
If ${\overline{\alpha}^*}^T$ is known, then a basic feasible
solution
 ${\alpha^*}^T=(\alpha^*_1,\alpha^*_2,\dots,\alpha^*_{r+s})$ \ of
system \ (\ref{blinn21}) \ that has
$\alpha^*_{r+s}=-\displaystyle\frac{1}{{\bf{w}}_{r,r+s}}t_r^*$ as a
basic component can be found as follows
\begin{equation}\label{jora2}
\!\!\!\!\alpha^*_i\!=\!\left\{%
\begin{array}{lll}
\quad\quad\quad\quad   \overline{\alpha}^*_i,
\quad\quad\quad\quad\quad\quad\quad\quad\quad\! i\!=\!1,2,\dots,k\!-\!1,r\!+\!1,\dots,r\!+\!s\!\!-\!\!1,\\[1mm]
\!-\!\sum\limits_{j=r+1}^{r+s-1}{\bf{w}}_{i,j}\overline{\alpha}^*_j\!+ \displaystyle\frac{{\bf{w}}_{i,r+s}}{{\bf{w}}_{r,r+s}}t_r^*+\!{\bf{w}}_{i,0}, \ i\!=\!k,k+1,\dots,r-1,\\[3mm]
\quad\quad\quad\quad       0,  \quad\qquad\qquad\qquad\qquad\
i\!=\!r;\\[1mm]
\quad\quad   -\displaystyle\frac{1}{{\bf{w}}_{r,r+s}}t_r^*,
\quad\quad\quad\qquad\quad\quad \ i\!=\!r+s,\\
\end{array}
\right.
\end{equation}
where $\alpha_r^*=0$ is a nonbasic component of ${\alpha^*}^T$. In
this case the inequality
\begin{equation}\label{nata3}\sum_{j=r}^{r+s-1}{\bf{w}}_{r,j}\alpha_j\leq
{\bf{w}}_{r,0}+t_r^*\end{equation} is weekly redundant for  system
(\ref{bingo28})  and an arbitrary inequality
\begin{equation}\label{druma1}\sum_{j=r+1}^{r+s-1}{\bf{w}}_{i,j}\alpha_j\leq{\bf{w}}_{i,0}+\displaystyle\frac{{\bf{w}}_{i,r+s}}{{\bf{w}}_{r,r+s}}t_r^*,\
\ i\in\{k,k+1,\dots,r\},
\end{equation}
 is redundant for system (\ref{bingo28}).

 If $k<r$ and the maximal value of the objective
function  $z_r^*=t_r^*< 0,$  then  system (\ref{blinn21}) has a
basic  feasible  solution
${\alpha^*}^T=(\alpha^*_1,\alpha^*_2,\dots,\alpha^*_{r+s})$ \  with
the basic component $\alpha^*_{r+s}$ if and only if system
(\ref{check01}) has a basic feasible solution
${\alpha^*}^T=(\alpha^*_1,\alpha^*_2,\dots,\alpha^*_{r+s})$ that
$\alpha^*_{r+s}$ is a basic component. If $k=r$ and the maximal
value of the objective function  $z_r^*=t_r^*< 0$, then system
(\ref{blinn21}) has no basic solution ${\alpha^*}^T$ that has
$\alpha^*_{r+s}$ as a basic component.
\end{theorem}
\begin{proof} The proof of the first part of this theorem is similar to the proof of Lemma
\ref{check3}. We consider the linear programming problem:
 \emph{Maximize  (\ref{doga1}) subject to
(\ref{blinn21})}. This problem has solutions and  the optimal value
of the objective function (\ref{doga1}) is equal to $t_r^*$. If we
dualize this problem, then we obtain the linear programming problem
(\ref{nata1}),(\ref{nata2}) for which the optimal value of objective
function (\ref{nata1}) is equal to $t_r^*$. As we have shown, in
this problem $\beta_i=0, \ i=k, k+1,\dots, r,$  because
${\bf{w}}_{i,r+s}<0, \ i=k,\  k+1, \ \dots, \ r$.
 Therefore problem (\ref{nata1}),(\ref{nata2}) can be written as the linear programming problem
(\ref{bostan1}),(\ref{bostan2}) that can be regarded as the dual
problem for the  linear programming problem (\ref{doga1}),
(\ref{check011}) and for the linear programming problem
(\ref{doga1}), (\ref{bingo28}). So, for these problems there exist
optimal solutions and the corresponding optimal values of the
objective functions are equal to $t_r^*$. Linear programming problem
(\ref{doga1}), (\ref{bingo28}) has an optimal solution that is
attained in an extreme point
$(\overline{\alpha}^*_{r+1},\overline{\alpha}^*_{r+2},\dots,\overline{\alpha}^*_{r+s-1})$
 of the set of solutions of system  (\ref{bingo28})  and the
optimal value of the objective function \ (\ref{doga1}) \ is equal
to \ $t_r^*$, i. e. \
$z_r^*=\sum\limits_{j=r+1}^{r+s-1}{\bf{w}}_{r,j}\overline{\alpha}_j^*
-{\bf{w}}_{r,0}=t_r^*$. This means that inequality \ (\ref{nata3}) \
is weekly redundant for system \ (\ref{bingo28}) \ and
${\overline{\alpha}^*}^T=(\overline{\alpha}^*_1,
\overline{\alpha}^*_2,\dots,\overline{\alpha}_{k-1}^*,\overline{\alpha}^*_r,\dots,\overline{\alpha}_{r+s-1}^*)$
\ with \
$\overline{\alpha}^*_i=-\sum\limits_{j=r+1}^{r+s-1}{\bf{w}}_{i,j}\overline{\alpha}_j^*
+{\bf{w}}_{r,0}, \ i=1,2,\dots,k-1,$  represents an optimal basic
solution of linear programming problem (\ref{doga1}),
(\ref{check011}). At the same time the inequality (\ref{nata3}) is
weakly redundant for  system (\ref{capjor1}) where
$(\overline{\alpha}_{r+1}^*,\overline{\alpha}_{r+2}^*,\dots,\overline{\alpha}^*_{r+s-1},\overline{\alpha}^*_{r+s})$
with \
$\overline{\alpha}^*_{r+s}=-\displaystyle\frac{1}{{\bf{w}}_{r,r+s}}t_r^*$
    represents an extreme point of the set of solutions of system
(\ref{capjor1})   in which the maximal value of the objective
function   (\ref{doga1})   is attained, i.e.
$z_r^*=\!\!\sum\limits_{j=r+1}^{r+s-1}\!\!{\bf{w}}_{r,j}\overline{\alpha}_j^*
-{\bf{w}}_{r,0}=t^*$. Thus, if
${\overline{\alpha}^*}^T=(\overline{\alpha}^*_{r+1},\overline{\alpha}^*_{r+2},\dots,\overline{\alpha}^*_{r+s-1})$
is an \ optimal \ basic \ solution  of \ linear programming problem
\ (\ref{doga1}), (\ref{check011}), then \ $(\overline{\alpha}^*_1, \
\overline{\alpha}^*_2, \ \dots \ \overline{\alpha}^*_{r+s-1}, \
\overline{\alpha}^*_{r+s})$, \ with \ \
$\alpha_{r+s}^*=-\displaystyle\frac{1}{{\bf{w}}_{r,r+s}}t_r^*$ \ \
and \
$\overline{\alpha}_i^*=-\sum\limits_{j=r+1}^{r+s-1}{\bf{w}}_{i,j}\overline{\alpha}^*_j
\ + \ \displaystyle\frac{{\bf{w}}_{i,r+s}}{{\bf{w}}_{r,r+s}}t_r^* \
+ \ {\bf{w}}_{i,0}\geq 0, \ \ \  i=k,k+1,\dots,r,$ represents an
optimal basic solution of linear programming problem (\ref{doga1})
(\ref{blinn21}), i.e.
${\alpha^*}^T=(\alpha^*_1,\alpha^*_2,\dots,\alpha^*_{r+s})$ with
components  determined according to (\ref{jora2}) represents a basic
feasible solution for system (\ref{blinn21}) that has
$\alpha_{r+s}^*=-\displaystyle\frac{1}{{\bf{w}}_{r,r+s}}t_r^*$ as a
basic component. Indeed, \ ${\alpha^*}^T$ \ contains no more than \
$k-1$ \ basic components of \ ${\overline{\alpha}^*}^T,$ no more
than $r-k$ components
$\alpha_i^*=\!-\!\sum\limits_{j=r+1}^{r+s-1}{\bf{w}}_{i,j}\overline{\alpha}^*_j\!+
\displaystyle\frac{{\bf{w}}_{i,r+s}}{{\bf{w}}_{r,r+s}}t_r^*+\!{\bf{w}}_{i,0},$
\ $i=k,k+1,\dots,r-1$, and the  basic component
$\alpha_{r+s}^*=\displaystyle\frac{1}{{\bf{w}}_{r,r+s}}t_r^*$, i.e.,
in the whole we obtain no more than $r$ basic components. Therefore
the inequality (\ref{nata3}) is weakly redundant for system
(\ref{bingo28}), and an arbitrary inequality (\ref{druma1}) is
redundant for system (\ref{bingo28}).

The proof of the second part of the theorem when  $z_r^*=t_r^*<0$
follows from Lemma \ref{check1}.
\end{proof}
\begin{remark}\label{fld1} Theorem \ref{capmt11} has been proven assuming that the upper bound of function (\ref{doga1}) on the set of solutions of system (\ref{blinn21}) is nonnegative finite value.
 In general, this
theorem  can be proven using such a bounded
 condition with respect \ to \ an \ arbitrary \ function \
$z_i=\sum_{j=r}^{r+s-1}{\bf{w}}_{i,j}\alpha_j- {\bf{w}}_{i,0,}$ \
$i\in\{k,k+1,\dots,r\}$. In the case when all these \ upper \ bounds
$z_i^*=t_i^*,$  $i=k, k+1,\dots, r,$ are negative,  system \
(\ref{blinn21}) \ has no  basic solution \
${\alpha^*}^T=(\alpha^*_1,\alpha^*_2,\dots,\alpha^*_{r+s})$ \ that
has $\alpha^*_{r+s}$ as a basic component. Additionally, if all
upper bounds $t_i^*, \ i\in\{k,k+1,\dots,r\}$ of the corresponding
functions $z_i$ \ $i\in\{k,k+1,\dots,r\}$ on the set of solutions of
system (\ref{blinn21}) are nonnegative finite values, then system
(\ref{blinn21}) \ may have \ different \ basic feasible solutions \
${\alpha^*}^T(i)=(\alpha^*_1(i),\alpha^*_2(i),\dots,\alpha^*_{r+s}(i))$
 with different   basic  components \
$\alpha^*_{r+s}(i)=-\displaystyle\frac{1}{{\bf{w}}_{i,r+s}}t_i^*, \
\ \ i\in\{k,k+1,\dots,r\}$. So, if consistent system (\ref{blinn21})
with $k\leq r$ has a basic feasible solution
${\alpha^*}^T=(\alpha^*_1,\alpha^*_2,\dots,\alpha^*_{r+s})$  that
$\alpha^*_{r+s}$  is a basic component, then such a basic solution
with such a basic component may be not unique.
\end{remark}

In the case when function (\ref{doga1}) is upper unbounded on the
set of solutions of system (\ref{blinn21}), the following theorem
holds.
\begin{theorem}\label{lbl1} Assume that system (\ref{blinn21}) with    $k<r$   has solutions and consider the linear programming problem: Maximize
the objective   function (\ref{doga1}) on the set of solutions of
system (\ref{blinn21}). If the objective function(\ref{doga1})  on
the set of solutions of system (\ref{blinn21}) is upper unbounded,
then  system (\ref{blinn21}) has no  basic feasible solution
${\alpha^*}^T=(\alpha^*_1,\alpha^*_2,\dots,\alpha^*_{r+s})$ that has
$\alpha_{r+s}^*$ as a basic component.
\end{theorem}
\begin{proof} We prove the theorem by contradiction, assuming
that the function (\ref{doga1}) on the set of solutions of system
(\ref{blinn21}) is upper unbounded and system (\ref{blinn21}) has a
basic feasible solution
${\alpha^*}^T=(\alpha^*_1,\alpha^*_2,\dots,\alpha^*_{r+s})$ that has
$\alpha_{r+s}^*$ as a basic component, where
$\alpha^*_{r+s}<\infty$. If it is so, then it easy to check that the
following system
 \begin{equation}\label{bcapmt1}
\!\!\left\{%
\begin{array}{l}
\alpha_1 \    \quad\quad\quad
+\sum\limits_{j=r+1}^{r+s-1}{\bf{w}}_{1,j}\alpha_j \quad\quad\quad\quad\quad \ \ \ \ \ \   =\ {\bf{w}}_{1,0},\\[1mm]
\ \ . \ \ . \ \ \ . \ \ .\ \ . \ \ . \ \ . \ \ . \ \ . \ \ . \ \ . \
\ . \ \ . \ \ . \ \ . \ \ . \ \ . \ \ . \ \ . \ \ . \ \ . \ \ . \ \ . \ \    \\[1mm]
\ \   \alpha_{k-1}\ \ \ \    +\sum\limits_{j=r+1}^{r+s-1}\!{\bf{w}}_{k-\!1,j}\alpha_j \ \quad\quad\quad\quad\quad \ \ \ \   ={\bf{w}}_{k-\!1,0},\\[2mm]
\quad\ \   \alpha_{k} \ \ \ \ +\sum\limits_{j=r+1}^{r+s-1}{\bf{w}}_{k,j}\alpha_j\ \ \ + \  {\bf{w}}_{k,r+s}\alpha_{r+s}  ={\bf{w}}_{k,0}\ \ - \ {\bf{w}}_{k,r+s}\alpha^*_{r+s},\\[1mm]
\quad \quad    \alpha_{k+1}\!+\!\sum\limits_{j=r+1}^{r+s-1}\!{\bf{w}}_{i+1,j}\alpha_j  +   {\bf{w}}_{k+1,r+s}\alpha_{r+s}\! ={\bf{w}}_{k+1,0}-{\bf{w}}_{k\!+\!1,r+s}\alpha^*_{r+s},\\[1mm]
\ \ . \ \ . \ \ . \ \ . \ \ . \ \ . \ \  .\ \ . \ \ . \ \ . \ \ . \
\ . \ \ . \ \ . \ \ . \ \ \ . \ \ .\ \ . \ \ . \ \ . \ \ . \ \ . \ \
. \ \ . \ \ . \ \ . \ \ . \ . \ \ . \\[0.5mm]
\quad \quad\quad\   \alpha_r\!+\sum\limits_{j=r+1}^{r+s-1}{\bf{w}}_{r,j}\alpha_j \ \ \ + \ {\bf{w}}_{r,r+s}\alpha_{r+s}  = {\bf{w}}_{r,0}\ \  - \ {\bf{w}}_{r,r+s}\alpha^*_{r+s},\\[1mm]
\qquad\quad\quad\quad\quad\quad\quad \alpha_j\geq 0, \quad j=1,\ 2,\
\dots,\ r+s,
\end{array}%
\right.
\end{equation}
has a basic feasible  solution
${\alpha^0}^T\!\!=\!(\alpha^0_1,\alpha^0_2,\dots,\alpha^0_{r+s-1},\alpha^0_{r+s})$
that has $\alpha^0_{r+s}=0$ as a degenerate basic component  and \
$\alpha^0_i=\alpha^*_i, \ \ i=1,2,\dots,r+s-1.$ Reversely also
holds, i.e. if for a given  $\alpha^*_{r+s}$  system (\ref{bcapmt1})
has a basic feasible solution
${\alpha^0}^T=(\alpha^0_1,\alpha^0_2,\dots,\alpha^0_{r+s-1},\alpha^0_{r+s})$
that has $\alpha^0_{r+s}=0$ as a degenerate basic component, \ then
system (\ref{blinn21}) has a basic feasible solution
${\alpha^*}^T=(\alpha^*_1,\alpha^*_2,\dots,\alpha^*_{r+s})$ that has
$ \alpha^*_{r+s}$ as a basic component and  $\alpha_i^*=\alpha_i^0,
\ i=1,2,\dots,r+s-1$.

So, ${\alpha^0}^T=(\alpha^0_1,\alpha^0_2,\dots,\alpha^0_{r+s})$ with
$\alpha^0_i=\alpha^*_i, \ i=1,2,\dots,r+s-1,$ and $\alpha^0_{r+s}=0$
represents  a basic feasible solution with degenerate basic basic
component $\alpha^0_{r+s}=0$ for system (\ref{bcapmt1}). Then
according to Theorem \ref{capmt1} there exists $i_0\in
\{k,k+1,\dots,r\}$ for which the inequality
\begin{equation}\label{aldr1}\sum\limits_{j=r+1}^{r+s-1}{\bf{w}}_{i_0,j}\alpha_j \leq
{\bf{w}}_{i_0,0}-{\bf{w}}_{i_0,r+s}\alpha^*_{r+s}
\end{equation}
 is weakly redundant for
system (\ref{bingo28}) and an arbitrary inequality
\begin{equation}\label{aldr2}\sum\limits_{j=r+1}^{r+s-1}{\bf{w}}_{i,j}\alpha_j
\leq {\bf{w}}_{i,0}-{\bf{w}}_{i,r+s}\alpha^*_{r+s}, \ i\in
\{k,k+1\dots,r\},\end{equation}
 is redundant for system
(\ref{bingo28}). This means that an arbitrary function $$z_{i}
=\sum\limits_{j=r+1}^{r+s-1}{\bf{w}}_{i,j}\alpha_j- {\bf{w}}_{i,0},
\ i\in \{k,k+1\dots,r\},$$ is upper bounded on the set of solutions
of \ system \ (\ref{bingo28}) \ and,  based on Theorem
\ref{capmt11},
 these functions are upper bounded on the set of solution
of system \ (\ref{blinn21}), i. e.,  we obtain a contradiction.
\end{proof}
\begin{remark} Theorem \ref{lbl1} \  has been proven for the case when \ the
\ function $z_r =\sum\limits_{j=r+1}^{r+s-1}{\bf{w}}_{rj}\alpha_j-
{\bf{w}}_{r,0}$ is unbounded on the set of solutions of system
(\ref{blinn21}). Obviously, this theorem is valid  also for the case
when an arbitrary function $z_{i_0}
=\sum\limits_{j=r+1}^{r+s-1}{\bf{w}}_{i_0,j}\alpha_j-
{\bf{w}}_{i_0,0}, \ i_0\in \{k,k+1\dots,r\}$ is unbounded on the set
of solutions of system (\ref{blinn21}).
\end{remark}

The presented above results prove the following theorem.
\begin{theorem}\label{corl81} The consistent system \
(\ref{blinn21})\ with \ $k\leq r$ \ has a basic feasible solution \
${\alpha^*}^T\!\!=\!(\alpha^*_1,\alpha^*_2,\dots,\alpha^*_{r+s})$ \
that has  $\alpha^*_{r+s}$ as a basic component if and only if there
exist \ $i_0\in \{k,k+1,\dots,r\}$ \ and \  $t_{i_0}^* \ (0\leq
t_{i_0}^*<\infty),$ \ such that $t^*_{i_0}$ is the upper bound of
the function
\begin{equation}\label{nics1}z_{i_0}
=\sum\limits_{j=r+1}^{r+s-1}{\bf{w}}_{i_0,j}\alpha_j-
{\bf{w}}_{i_0,0}\end{equation} on the set of solutions of system
(\ref{blinn21}). In this case system \ (\ref{blinn21})\ with \
$k\leq r$  has a basic feasible solution \
${\alpha^*}^T\!\!=\!(\alpha^*_1,\alpha^*_2,\dots,\alpha^*_{r+s})$, \
where \
$\alpha^*_{r+s}=-\displaystyle\frac{1}{{\bf{w}}_{i_0,r+s}}t_{i_0}^*$
\  is a basic component. Such a basic feasible solution for system
system (\ref{blinn21}) can be found by using an optimal basic
solution
${\overline{\alpha}^*}^T=(\overline{\alpha}^*_1,\overline{\alpha}^*_2,\dots,\overline{\alpha}^*_{i_0-1},\overline{\alpha}^*_{i_0+1},\dots,
\overline{\alpha}_{r+s-1}^*)$ \ of the following linear programming
problem: Maximize the objective function (\ref{nics1}) on the set of
solution of system (\ref{check011}). If such an optimal basic
solution for this problem is known then the components of the basic
solution ${\alpha^*}^T$ for system (\ref{blinn21}) that has
$\alpha^*_{r+s}=-\displaystyle\frac{1}{{\bf{w}}_{i_0,r+s}}t_{i_0}^*$
as a basic component can be found as follows
\begin{equation}\label{jora02}
\!\!\!\!\alpha^*_i\!=\!\left\{%
\begin{array}{lll}
\quad\quad\quad\quad   \overline{\alpha}^*_i,
\quad\quad\quad\quad\quad\quad\quad\quad\quad\! i\!=\!1,2,\dots,k\!-\!1,r\!+\!1,\dots,r\!+\!s\!\!-\!\!1,\\[1mm]
\!-\!\sum\limits_{j=r+1}^{r+s-1}{\bf{w}}_{i,j}\overline{\alpha}^*_j\!+ \displaystyle\frac{{\bf{w}}_{i,r+s}}{{\bf{w}}_{r,r+s}}t_{i_0}^*+\!{\bf{w}}_{i,0}, \ i\!=\!k,k+1,\dots,r-1,\\[3mm]
\quad\quad\quad\quad       0,  \quad\qquad\qquad\qquad\qquad\
i\!=\!r;\\[1mm]
\quad\quad   -\displaystyle\frac{1}{{\bf{w}}_{r,r+s}}t_{i_0}^*,
\quad\quad\quad\qquad\quad\quad \ i\!=\!r+s,\\
\end{array}
\right.
\end{equation}
where $\alpha_{i_0}^*=0$ is a nonbasic component of ${\alpha^*}^T$.
 Additionally, the inequality
$$\sum_{j=r}^{r+s-1}{\bf{w}}_{i_0,j}\alpha_j\leq
{\bf{w}}_{i_0,0}+t_{i_0}^*,$$ is weekly redundant for system
(\ref{bingo28}) and an arbitrary inequality
$$\sum_{j=r+1}^{r+s-1}{\bf{w}}_{i,j}\alpha_j\leq{\bf{w}}_{i,0}+\displaystyle\frac{{\bf{w}}_{i,r+s}}{{\bf{w}}_{i_0,r+s}}t_{i_0}^*,\
\ i\in\{k,k+1,\dots,r\},$$
 is redundant for system (\ref{bingo28}).
 \end{theorem}
In fact the  second part of this theorem represents Theorem
\ref{capmt11} for  the case $i_0=r$.

\medskip
To check if a function
$z_i=\sum\limits_{j=r+1}^{r+s-1}{\bf{w}}_{i,j}\alpha_j-{\bf{w}}_{i,0},$
\ $i\in \{k,k+1,\dots,r\}$, on the set of solutions of system
(\ref{blinn21}) is upper bounded or it is upper unbounded on the set
of solutions of this system  the following lemma can be used.
\begin{lemma}\label{finis1} Let the consistent system (\ref{blinn21}) with
$k\leq r$ be given. Then a function
\begin{equation}\label{bneq1}
z_{i_0}=\sum\limits_{j=r+1}^{r+s-1}{\bf{w}}_{i_0,j}\alpha_j-{\bf{w}}_{i_0,0},
\ i_0\in \{k,k+1,\dots,r\}
\end{equation}
is upper bounded on the set of solutions of system (\ref{blinn21})
if an only if the following system
\begin{equation}\label{stefan1}
\left\{%
\begin{array}{ll}\displaystyle\sum_{j=r+1}^{r+s-1}{\bf{w}}_{i,j}\alpha_j\qquad\qquad\qquad\leq0,
\ \ i=1,2,\dots,k-1,\\[1mm]
\displaystyle\sum_{j=r+1}^{r+s-1}{\bf{w}}_{i,j}\alpha_j+{\bf{w}}_{i,r+s}\alpha_{r+s}\leq0,\
\ i=k,k+1,\dots,r,\\[1mm]
\displaystyle\sum\limits_{j=r+1}^{r+s-1}{\bf{w}}_{i_0,j}\alpha_j = 1,\\[1mm]
\quad\quad\quad\quad\ \  \alpha_j\geq 0, \quad\quad\quad
j=r+1,r+2,\dots,r+s.
\end{array}%
\right.
\end{equation}
has no solutions. If system (\ref{stefan1}) has solutions then the
function (\ref{bneq1}) is upper unbounded on the set of solutions of
system (\ref{blinn21}).
\end{lemma}
\begin{proof} $\Rightarrow$ Assume that the function (\ref{bneq1}) for a given $i_0\in \{k,k+1,\dots,r\}$ is upper bounded on the set of solutions of system
(\ref{blinn21}). Then there exists the optimal solution of the
linear programming problem: Maximize the objective function
(\ref{bneq1}) subject to (\ref{blinn21}). This means that  the
linear function
$$z'_{i_0}=\sum_{j=r+1}^{r+s-1}{\bf{w}}_{i_0,j}\alpha_j$$
is upper bounded on the set of solutions of system (\ref{capjor1})
and according to Lemma 1.10 from \cite{che} (page 96) for an
arbitrary solution of the homogeneous system
\begin{equation}\label{stefan3}
\left\{%
\begin{array}{ll}\displaystyle\sum_{j=r+1}^{r+s-1}{\bf{w}}_{i,j}\alpha_j\qquad\qquad\qquad\leq0,
\ \ i=1,2,\dots,k-1,\\[1mm]
\displaystyle\sum_{j=r+1}^{r+s-1}{\bf{w}}_{i,j}\alpha_j+{\bf{w}}_{i,r+s}\alpha_{r+s}\leq0,\
\ i=k,k+1,\dots,r,\\[1mm]
\quad\quad\quad\quad\ \  \alpha_j\geq 0, \quad\quad\quad
j=r+1,r+2,\dots,r+s.
\end{array}%
\right.\end{equation} we have
$\displaystyle\sum\limits_{j=r+1}^{r+s-1}{\bf{w}}_{i_0,j}\alpha_j\leq
0,$ i.e. system (\ref{stefan1}) has no solutions.

$\Leftarrow$ Assume that the function (\ref{bneq1}) for a given
$i_0\in \{k,k+1,\dots,r\}$ is upper unbounded on the set of
solutions of system (\ref{blinn21}), and let us show that system
(\ref{stefan1}) has solutions. Indeed, if function (\ref{bneq1}) is
upper unbounded on the set of solutions of system (\ref{blinn21}),
then  for an arbitrary positive value $N$, system (\ref{capjor1})
has solutions such that
$\sum\limits_{j=r+1}^{r+s-1}{\bf{w}}_{i_0,j}\alpha_j-{\bf{w}}_{i_0,0}\geq
N$ and this involves that the set of solutions of the system
\begin{equation}\label{stefan4}
\left\{%
\begin{array}{ll}\quad\displaystyle\sum_{j=r+1}^{r+s-1}{\bf{w}}_{i,j}\alpha_j\ \quad\quad\quad\qquad\qquad\qquad\quad\leq{\bf{w}}_{i,0},
\ \ i=1,2,\dots,k-1,\\[1mm]
\quad\displaystyle\sum_{j=r+1}^{r+s-1}{\bf{w}}_{i,j}\alpha_j+{\bf{w}}_{i,r+s}\alpha_{r+s}\quad\quad\quad\quad\
\leq{\bf{w}}_{i,0},\ \ i=k,k+1,\dots,r,\\[1mm]
-\displaystyle\sum\limits_{j=r+1}^{r+s-1}{\bf{w}}_{i_0,j}\alpha_j\quad\quad\quad\quad\quad\quad+\alpha_{r+s+1}\leq -{\bf{w}}_{i_0,0},\\[1mm]
\quad\quad\quad\quad\ \  \alpha_j\geq 0, \quad\quad\quad
j=r+1,r+2,\dots,r+s+1 ,
\end{array}%
\right.
\end{equation}
is unbounded with respect to the positive direction of coordinate
$\alpha_{r+s+1}$.  Therefore, according to Theorem 1.13 and Remark 2
from \cite{che} (pages 93,96), the set of solutions of
 system (\ref{stefan4}) is unbounded with respect to the positive direction of coordinate
 $\alpha_{r+s+1}$ if and only if the homogeneous system
\begin{equation}\label{stefan5}
\left\{%
\begin{array}{ll}\quad\displaystyle\sum_{j=r+1}^{r+s-1}{\bf{w}}_{i,j}\alpha_j\ \quad\quad\quad\qquad\qquad\qquad\quad\leq0,
\ \ i=1,2,\dots,k-1,\\[1mm]
\quad\displaystyle\sum_{j=r+1}^{r+s-1}{\bf{w}}_{i,j}\alpha_j+{\bf{w}}_{i,r+s}\alpha_{r+s}\quad\quad\quad\quad\
\leq0,\ \ i=k,k+1,\dots,r,\\[1mm]
-\displaystyle\sum\limits_{j=r+1}^{r+s-1}{\bf{w}}_{i_0,j}\alpha_j\quad\quad\quad\quad\quad\quad+\alpha_{r+s+1}\leq0,\\[1mm]
\quad\quad\quad\quad\ \  \alpha_j\geq 0, \quad\quad\quad
j=r+1,r+2,\dots,r+s+1 .
\end{array}%
\right.
\end{equation}
is unbounded with respect to the positive direction of coordinate
$\alpha_{r+s+1}$ of the set of  solutions of system (\ref{stefan5}).
So, system (\ref{stefan1}) has solutions.
\end{proof}

Based on results obtained in this section, in the next section we
propose a polynomial algorithm for checking if $Y\not\subset Y_h$.

\section{\!Polynomial algorithm to check if the \ condition \ $Y\not\subset
Y_h$\\ holds for a given $h\in [-2^L,2^L]$}\label{ldmd3} Let us
assume that the coefficients of  disjoint bilinear programming
problem (\ref{eq01})-(\ref{eq03}) with a perfect disjoint subset $Y$
determined by the set of solutions of system (\ref{eq03}) are
integer. In this section, based on the results of the previous
section, we propose a polynomial algorithm for checking if system
(\ref{sure1}) has a basic solution $x_1^*,
x_2^*,\dots,x_{n+q}^*,v_1^*, v_2^*,\dots,v_{p+1}^*$ that has
$v_{p+1}^*$ as basic component for a fixed $h\in [-2^L,2^L]$. Note
that in system (\ref{sure1}), parameter  $h$ may be not integer,
however it is a rational value represented by an irreducible
rational fraction $h=\displaystyle\frac{M}{N}$, where $|M|$ and
$|N|$ do not exceed $2^L$. Therefore,  by multiplying the
coefficients of equation $r=q+m+1$ of system  (\ref{sure1}) by a
suitable integer value of the form $2^L$, all coefficients of this
equation become integer.

  According to Theorem \ref{thm101} and Corollary
\ref{blin0},  if system (\ref{eq1710}) has no solutions, then for a
given $h\in [-2^L,2^L]$, the property $Y \not \subset Y_h$ holds if
and only if system (\ref{sure1}) has a basic feasible solution that
has $v_{p+1}$  as a basic component. In the case when system
(\ref{eq1710}) has solutions, the problem of checking the condition
$Y \not \subset Y_h$ can be easily solved on the basis of Corollary
\ref{dimop} and Theorem \ref{tm01}. Moreover, in this case, the
optimal value $h^*$ of the disjoint bilinear programming problem
(\ref{eq01})-(\ref{eq03}) and an optimal point $x^*\in X$ for this
problem can be found by solving the linear programming problem
(\ref{crd})-(\ref{eq1710}). Therefore here we propose a polynomial
algorithm for checking the existence of a basic solution for system
(\ref{sure1}) that has $v_{p+1}$ as a basic component  when system
(\ref{eq1710}) has no solutions. In this case system (\ref{sure1})
is consistent for an arbitrary $h\in [-2^L,2^L]$, and the proposed
algorithm can be applied for checking if $Y \not \subset Y_h$. In
this algorithm we assume that system (\ref{sure1}) is represented in
 matrix form by system (\ref{reneq1}) consisting of $r$ linear
equations with nonnegative conditions for the vector of variables
${\alpha}^T=(\alpha_1,\alpha_2,\dots, \alpha_{r+s})$, where
$r=q+m+1, \ s=n+p-m$ and the variable $\alpha_{r+s}$ corresponds to
the variable $v_{p+1}$ in system \ (\ref{sure1}). Thus, the proposed
algorithm determines if either the property $Y \not \subset Y_h$
holds or fail to holds. In the case when for a given $h\in
[-2^L,2^L]$ the condition $Y \not \subset Y_h$ holds the algorithm
determines  a basic feasible solution
${\alpha^*}^T=(\alpha^*_1,\alpha^*_2,\dots,\alpha^*_{r+s})$ that has
$\alpha_{r+s}^*$ as a basic component and consequently algorithm
determines a basic solution $x_1^*,\ x_2^*,\ \dots,\
x_{n+q}^*,v_1^*,\ v_2^*,\ \dots,v_{p+1}^*$ for system (\ref{sure1})
that has $v_{p+1}^*$ as a basic component.

\bigskip \noindent\textbf{Algorithm 1:}

\bigskip
\noindent  \emph{Step 1 (Preliminary step):}\\
Fix \  $h\in [-2^L,\  2^L]$ \  and consider system \ (\ref{sure1}) \
with the given \ $h$. In  matrix form this system  is represented by
system \ (\ref{reneq1}) \ with respect to the vector of variables
${\alpha}^T=(\alpha_1,\alpha_2,\dots, \alpha_{r+s})$, where the
component \ $\alpha_{r+s}$ \ in \ (\ref{reneq1}) \ corresponds to
component \ $v_{p+1}$ \  in (\ref{sure1}), i.e. \
$\alpha_{r+s}=v_{p+1}$.  If the conditions  of Theorem \ref{thm101}
are satisfied, then for a given \ $h\in[2^{-L}, 2^L]$, \ system \
(\ref{reneq1}) \ has solutions.

\medskip
\noindent \emph{Step 2:}\\
\noindent Find a basic feasible solution
${\alpha^0}^T=(\alpha_1^0,\alpha_2^0,\dots,\alpha_{r+s}^0)$ of
system   (\ref{reneq1}). If $\alpha_{r+s}^0$ is a  basic component
of ${\alpha^0}^T$, then, according  to Theorem \ref{thm101}, system
(\ref{sure1}) has a basic  feasible  solution that has
$v_{p+1}^0=\alpha_{r+s}^0$ as a basic component,  i. e.\emph{ $Y
\not \subset Y_h$ holds and
 Stop}. If  $\alpha_{r+s}^0$ is a nonbasic component
 of ${\alpha^0}^T$, then all basic components are among
$\alpha_1^0,\alpha_2^0,\dots,\alpha_{r+s-1}^0$. Without loss of
generality,  we  may  assume  that  the  basic  components of
${\alpha^0}^T$ represent the first its $r$ components
$\alpha_1^0,\alpha_2^0,\dots,\alpha_{r}^0$, which correspond to the
linearly independent column vectors \ $W_1,W_2,\dots,W_r$  that form
a basis for the set of column  vectors of matrix  $W$  in
(\ref{reneq1});  if this is not so,  then we can relabel variables
in (\ref{reneq1})   so that the first $r$
 variables in  (\ref{reneq1}) will
 correspond to the basic variables. In this case, system  (\ref{reneq1})  in
 extended form has the structure of system  (\ref{blin2}) with the first  $r$
variables corresponding to the basic ones and with these  basic
variables  expressed  explicitly via  the  nonbasic variables, where
$\alpha_{r+s}=v_{p+1}$. Let ${\bf{W}}_{r+s}$ be the column vector of
coefficients
 of  system  (\ref{blin2})  that
corresponds  to  $\alpha_{r+s}$, i. e.
${\bf{W}}_{r+s}^T=({\bf{w}}_{1,r+s},{\bf{w}}_{2,r+s},\dots,{\bf{w}}_{r,r+s})$,
and ${\bf{W}}_0$ \ be the column vector  of  right-hand  side of
coefficients  of \ system  (\ref{blin2}), i.e.
${\bf{W}}_0^T=({\bf{w}}_{1,0},{\bf{w}}_{2,0},\dots,{\bf{w}}_{r,0}).$
 If  the  components  of  these  vectors
 satisfy conditions of  Lemma \ref{izar1}, i.e. if either
 ${\bf{W}}_{r+s}$ contains a positive
 component  ${\bf{w}}_{i_0,r+s}\ ({\bf{w}}_{i_0,r+s}>0)$   or
${\bf{W}}_{r+s}$ \ contains a nonzero  component
${\bf{w}}_{i_0,r+s}\neq 0$  (positive or negative)  such that  the
corresponding  component ${\bf{w}}_{i_0,0}$  of  vector
${\bf{W}}_0^T$  is equal to  zero,  then according to Lemma
\ref{izar1}, system (\ref{blin2})   has a  basic  feasible  solution
${\alpha^*}^T=(\alpha^*_1,\alpha^*_2,\dots,\alpha^*_{r+s})$ that has
$\alpha^*_{r+s}$ as a basic component; such a basic solution can be
obtained by using one step of the  standard pivoting calculating
procedure from  Lemma \ref{izar1}. So, in this case, system
(\ref{sure1}) has a basic feasible solution $x_1^*,\ x_2^*,\ \dots,\
x_{n+q}^*,v_1^*,\ v_2^*,\ \dots,v_{p+1}^*$ that has $v_{p+1}^*$ as a
basic component, i. e. \emph{ $Y \not \subset Y_h$ holds and Stop}.
If the components of the vectors ${\bf{W}}_{r+s}$ and ${\bf{W}}_0$
of system (\ref{blin2}) do not satisfy the conditions of Lemma
\ref{izar1}, then go to  next \emph{Step 3}.

\medskip\noindent
 \emph{Step 3:}\\
We relabel the equations and variables
$\alpha_1,\alpha_2,\dots,\alpha_r$ of system (\ref{blin2}) so that
it has the form of  system (\ref{blinn21}), where \
$${\bf{w}}_{i,r+s}=0,\ \ i=1,2,\dots, k-1; \ \ \ \  {\bf{w}}_{i,r+s}<0,\
\ \ i=k, k+1,\dots, r,$$ and
$$ \ {\bf{w}}_{i,0} \ \geq \ 0,\ \ \ i=1,2,\dots,k-1;\ \ \ \
\ \ {\bf{w}}_{i,0}\ >  0,\ \ \  i=k,k+1,\dots,r.$$ Then go to next
\emph{Step 4}.

\medskip\noindent
 \emph{Step 4:}\\
The problem of the existence of the basic  solution
${\alpha^*}^T=(\alpha^*_1,\alpha^*_2,\dots,\alpha^*_{r+s})$ for
system (\ref{blinn21}) that has $\alpha^*_{r+s}$ as a basic
component  can be solved on the basis  of  Theorems
\ref{capmt11}-\ref{corl81} and  Lemma \ref{finis1}  as  follows. Fix
an  index $i_0\in\{k, k+1, \dots, r\}$  and check if system \
(\ref{stefan1}) \ is consistent. If this system has solutions, then
based on Lemma \ref{finis1} and Theorem \ref{lbl1}, \ system
(\ref{blinn21}) has no a basic solution with the basic component \
$\alpha^*_{r+s}$ \ and consequently system \ (\ref{sure1}) \ has no
basic solution that has \ $v^*_{p+1}$, \ as a basic component, \
i.e.\ the \ \emph{condition \ $Y \not \subset Y_h$ \ fail to holds
and Stop}; \ if \ system \ (\ref{stefan1}) \ has no solutions, then
we consider for each \ $i_0\in \{k,k+1,\dots,r\}$ \ the \ linear \
programming \ problem: Maximize the objective function \
(\ref{bneq1}) \ subject to \ (\ref{blinn21}). By solving such
linear\  programming \ problems \ for \ $i_0= k,k+1,\dots,r$, we
determine the corresponding optimal values
$z^*_k,z^*_{k+1},\dots,z^*_r$ of the objective functions for these
linear programming problems. If there exists \ $z_{i_0}^*\in
\{z^*_k,z^*_{k+1},\dots,z^*_r\}$ \ such that \ $z_{i_0}\geq 0$, \
then, according to Theorem \ref{capmt11}, system \ (\ref{blinn21}) \
has a basic solution \
${\alpha^*}^T=(\alpha^*_1,\alpha^*_2,\dots,\alpha^*_{r+s})$ \ such
that $\alpha^*_{r+s}$ is a basic component and consequently  system
\ (\ref{sure1}) \ has a basic  solution that has \ $v^*_{p+1}$, as a
basic component,  i.e. \ \emph{$Y \not \subset Y_h$ \ holds \ and
Stop}; if $z^*_i<0, \ i=k,k+1,\dots,r,$ then according to Corollary
\ref{clly}, system (\ref{blinn21}) has no  basic solution
${\alpha^*}^T=(\alpha^*_1,\alpha^*_2,\dots,\alpha^*_{r+s})$ that has
$\alpha^*_{r+s}$ as a basic component, and consequently system
(\ref{sure1}) has no  basic  solution that has $v^*_{p+1}$ as a
basic component, i.e. \emph{ condition $Y \not \subset Y_h$ fail to
holds and Stop}.
\begin{remark} \ Algorithm 1 determines if  the condition
\ $Y \not \subset Y_h$ \ holds for   a  given   $h\in [-2^L,2^L]$.
In the case when this condition holds, the algorithm determines also
a basic feasible solution
${\alpha^*}^T=(\alpha^*_1,\alpha^*_2,\dots,\alpha^*_{r+s})$  of
system (\ref{blinn21}) that has $\alpha^*_{r+s}$ as a basic
component for a given $h$, and consequently determines a basic
feasible solution $x_1^*,\ x_2^*,\ \dots,\ x_{n+q}^*,v_1^*,\ v_2^*,\
\dots,v_{p+1}^*$ of system (\ref{sure1}) that has $v_{p+1}^*$ as the
basic component.
\end{remark}

In the worst case Algorithm 1 solve not more that $r$ linear
programming problems. Therefore the computational complexity of this
algorithm depends on the computational complexity of linear
programming algorithms \ used \ in  Algorithm 1. Taking into account
that for solving a linear programming problem  there exist
polynomial algorithms (see \cite{dant1,dant2,kmkr,k80,k82,sh1}), we
can formulate the following result.
\begin{theorem} Algorithm 1 for checking  the condition $Y \not \subset Y_h$ and determining a basic feasible solution
$x_1^*,\ x_2^*,\ \dots,\ x_{n+q}^*,v_1^*,\ v_2^*,\ \dots,v_{p+1}^*$
of system (\ref{sure1}) that has  $v_{p+1}^*$  as a basic component
for fixed $h\in [-2^L,2^L]$ is polynomial.
\end{theorem}
\section{Polynomial algorithm for solving the disjoint bilinear
programming problem (\ref{eq01})-(\ref{eq03})}\label{finita} We
present a polynomial algorithm for solving the  disjoint bilinear
programming problem (\ref{eq01})-(\ref{eq03}) in which the set of
solutions $Y$ of system (\ref{eq03}) satisfies condition a)-c), i.e.
when $Y$ has the structure of a perfect polytope.

 \medskip
\noindent \textbf{Algorithm 2:}

\medskip\noindent If for disjoint bilinear programming  problem \
(\ref{eq01}) - (\ref{eq03}) the corresponding system (\ref{eq1710})
has solutions, then an optimal point $x^*\in X$ for this problem can
be found by solving the following  linear programming problem:
Minimize the objective function (\ref{crd}) subject to
(\ref{eq1710}). Then $x^*$ together with an arbitrary $y\in Y$
represent an optimal solution of disjoint bilinear programming
problem (\ref{eq01})-(\ref{eq03}). In this case, the optimal value
of the disjoint bilinear programming problem is $h^*=gx^*$. If
system (\ref{eq1710}) is inconsistent, then we use the following
iterative calculation procedure for determining the optimal solution
of disjoint bilinear programming problem \ (\ref{eq01}) -
(\ref{eq03}):

\medskip
\noindent\emph{Step 0:(Preliminary step }:\\
Fix $h_1^1=-2^L$,  \ $h_1^2= 2^L$, \ $\varepsilon = 2^{-2L-2}$.

\medskip
\noindent\emph{General  step (step $k, \, k\geq 1$):}\\
 Check if
 $h_k^2 -h_k^1\leq \varepsilon$.  If  $h_k^2 -h_k^1\leq
\varepsilon$, then go to \emph{Final step}; If $h_k^2 -h_k^1
>\varepsilon,$ then  find  $h^0_k=\displaystyle\frac{h_k^1 +
h_k^2}{2}$,  and  consider system \ (\ref{reneq1}) \ with \
$h=h^0_k$, \ and  apply Algorithm 1 to determine if this system has
a basic feasible solution
${\alpha^*}^T=(\alpha^*_1,\alpha^*_2,\dots,\alpha^*_{r+s})$ that has
$\alpha^*_{r+s}$ as the basic component. If system (\ref{reneq1})
has such a solution, then $Y \not \subset Y_h$ and based on the
relationship of this system with system (\ref{sure1}), we obtain a
basic feasible solution $x_1^*,x_2^*,\dots,x_{n+q}^*,v_1^*,v_2^*,
\dots,v_{p+1}^*$ for system (\ref{sure1}) with $h=h^0_k$ that has
$v_{p+1}^*$ as a basic component. In this case, we set
$h_{k+1}^1=h^1_k$, $h_{k+1}^2=h^0_k$ and go to step $k+1$. If system
(\ref{reneq1}) has no a basic feasible solution
${\alpha^*}^T=(\alpha^*_1,\alpha^*_2,\dots,\alpha^*_{r+s})$ that has
$\alpha^*_{r+s}$ as a basic component, then the condition $Y \not
\subset Y_h$ fails to hold, and  we put $h^1_{k+1}=h^0_k$,
$h^2_{k+1}=h^2_k$ and go to step $k+1$.

\medskip\noindent
\textit{Final Step}:\\
After $3L+2$ iterations of the general step of the algorithm, we
determine the value $\widetilde h=h_{3L+2}$ with no more $L$ digits
before the decimal point and no more $3L+2$ digits after the decimal
point. The value $\widetilde h$  approximates the exact optimal
value $h^*$ that is a rational value representing an irreducible
fraction $h^*=\displaystyle\frac{M}{N}$ such that $|M|,|N|\leq 2^L$
and $|\widetilde h -\displaystyle\frac{M}{N}|\leq 2^{-2L-2}$.
Therefore, here, in a similar way as it has been done in
\cite{k80,k82} for the linear programming problem, we find the exact
value $h^*$ by representing \ $\widetilde h$ as an infinite fraction
expansion. After that, fix $h=h^*$  in (\ref{sure1}) and find an
optimal basic solution $x_1^*,\ x_2^*,\ \dots,\ x_{n+q}^*,v_1^*,\
v_2^*,\ \dots,v_{p+1}^*$  for system (\ref{sure1}) that has
$v_{p+1}^*$ as the basic component. According to Theorem
\ref{cor102}, \ $x_1^*,\ x_2^*,\ \dots,\ x_n^*$ \ represent the
components of an optimal point $x^*$ for the disjoint linear
programming problem (\ref{eq01})-({\ref{eq03}). If $x^*$ is known,
then by finding a solution of system (\ref{ghdm1}) we determine
 the components $y^*_1,y^*_2,\dots,y^*_m$ of the
optimal point $y^*$, where $(x^*,y^*)$ is an optimal solution of
disjoint bilinear programming (\ref{eq01})-{\ref{eq03} for which
$z^*=x^*Cy^*+gx^*+ey^*$ and \emph{Stop}.

\medskip
Finally, we can note that the computational complexity of  Algorithm
2 for solving problem (\ref{eq01}) - (\ref{eq03}) depends in the
most part on the computational complexity of the linear programming
algorithms used at the general  and  final steps of the algorithm.
Taking into account that for the linear programming problem there
exist polynomial algorithms  (see
\cite{dant1,dant2,kmkr,k80,k82,sh1}) and that the number of
iterations at the general step of the algorithm is polynomial, we
can formulate the following theorem.
\begin{theorem} Algorithm 2 finds an optimal solution of the disjoint
bilinear programming problem (\ref{eq01})-(\ref{eq03}) in polynomial
time.
\end{theorem}
\section{Application of Algorithm 2 for solving the problems with an acute-angled
polytope from Section \ref{ddd}}\label{finita1} In this section, we
show how Algorithm 2 can be applied for solving the boolean linear
programming problem (\ref{eq4})-({\ref{eq5}) and the piecewise
linear concave programming problem (\ref{eq2}), (\ref{eq9})}.
\subsection{Application of Algorithm 2 \ for \ finding  a  boolean \ solution \ of system
\ (\ref{eq31})\  and solving the boolean linear programming \
(\ref{eq4})-(\ref{eq5})}The problem of finding a boolean solution
for system (\ref{eq31}) can be solved by using Algorithm 2, because
this problem can be represented as the disjoint bilinear programming
problem (\ref{equat3})-(\ref{equat5}) where the set of solution $Y$
has the structure of an acute-angled polytope. Indeed, in this case
the matrix $D=(d_{ij})_{p\times n}$ and the vector
$d_0=(d_{10},d_{20},\dots, d_{m0})^T$ have the following structure
$$D=\left(\begin{array}{ccccc}
\ \ 1 &\ \ 0 &\ \ 0 & \cdots &\ \ 0\\
\ \ 0 &\ \ 1 &\ \ 0 & \cdots &\ \ 0\\
\ \ 0 &\ \ 0 &\ \ 1 & \cdots &\ \ 0\\
\vdots & \vdots & \vdots & \ddots & \vdots \\
\ \ 0 &\ \ 0 &\ \ 0 & \cdots &\ \ 1\\
-1 & \ \ 0 &\ \ 0 & \cdots &\ \ 0\\
\ \ 0 & -1 &\ \ 0 & \cdots &\ \ 0\\
\ \ 0 &\ \ 0 & -1 & \cdots &\ \ 0\\
\vdots & \vdots & \vdots & \ddots & \vdots \\
\ \ 0 &\  \ 0 &\ \ 0 & \cdots & -1
\end{array}\right), \quad
d_0=\left(\begin{array}{c}
1\\
1\\
1\\
\vdots\\
1\\
0\\
0\\
0\\
\vdots\\
0
\end{array}\right),
$$
where $p=2n$  and \ $C=(c_{ij})_{n\times m}$, \ $g=(g_1, g_2,\dots,
g_n),$ \ $ e=(e_1, e_2,\dots, e_m)$ are defined as follows:
$$
c_{ij}=\left\{%
\begin{array}{ll}
            2, &  \mbox{if } \  i=j, \\[1.5mm]
            0, &  \mbox{if } \  i\neq j,
\end{array}
\right. \  g_i=-1,\  e_j=-1, \ \ i=1,2,\dots n; j=1,2,\dots,m,
$$
Therefore by applying Algorithm 2, we can determine a boolean
solution for the system of linear inequalities (\ref{eq31}) if the
optimal value of the object function $h^*$ of the considered
disjoint bilinear programming problem is equal to zero; otherwise
this system has no a boolean solution.

If it is necessary to solve the classical boolean linear programming
problem (\ref{eq4}),(\ref{eq5}) then we can use the disjoint
bilinear programming problem (\ref{eq6})-(\ref{eq8}) with the
corresponding $M$ chosen as mentioned in Section \ref{ddd}. In this
case the matrix $D$  and column vector $d_0$ are the same as above,
and $C=(c_{ij})_{n\times m}$, $g=(g_1, g_2,\dots, g_n),$ $ e=(e_1,
e_2,\dots, e_m)$ are defined as follows:
$$
c_{ij}=\left\{%
\begin{array}{ll}
            2M, &  \mbox{if } \  i=j, \\[1.5mm]
             0, &  \mbox{if } \  i\neq j,
\end{array}
\right. \  g_i=c_i-M,\  e_j=-M, \ \ i=1,2,\dots n; j=1,2,\dots,m.
$$
\subsection{Polynomial time \ algorithm \ for \ solving
\  the \  piecewise \ linear  concave programming problem
(\ref{eq2}),(\ref{eq9})}\label{finita2} The piecewise linear concave
programming problem consists of minimizing the  piecewise linear
concave function (\ref{eq9}) on the set of solution of system
(\ref{eq2}). In Section \ref{ddd} has been shown that this problem
can be represented as the bilinear programming problem
(\ref{eq10})-(\ref{eq12}). It can be observed that in this problem
the set of solutions  $Y$ of system (\ref{eq12}) has the structure
of an acute-angled polytope. So, we obtain a disjoint bilinear
programming problem (\ref{eq01})-(\ref{eq03}) in which the matrix
$D=(d_{ij})_{p\times m}$ and vector
$d_0=(d_{10},d_{20},\dots,d_{p0})^T$ are defined as follows:
$$
D=\left(\begin{array}{ccccccccccccc}
 \ 1 & \ 1 & \dots & 1 &  0 &  0 & \dots &\ 0 & \dots & \ 0 \ & \  0 & \dots & \  0\\
 \ 0 & \ 0 & \dots &  0 &  1 &  1 & \dots &\ 1 & \dots & \  0 \ & \  0 & \dots & \  0\\
\vdots & \vdots & \ddots & \vdots & \vdots & \vdots & \ddots &
\vdots & \ddots & \vdots & \vdots & \ddots & \vdots \\
 \ 0 & \  0 & \dots &   0 &   0 &  0 & \dots &\   0 & \dots &\   1\  &\   1 & \dots &\  1\\
 -1  & \ \   0 \ \ & \ \dots \  & \  0 \   & \   0 \   & \   0\   & \dots &\  0 \ & \dots & \  0 & \  0 & \dots &\   0\\
 \  0 & -1& \dots &   0 &   0 &   0 & \dots &\   0 & \dots &\   0 \ &\   0 & \dots &\ 0\\
\vdots & \vdots & \ddots & \vdots & \vdots & \vdots & \ddots &
\vdots & \ddots & \vdots & \vdots & \ddots & \vdots \\
 \  0 &\   0 & \dots &  0 &   0 &  0 & \dots &\   0 & \dots & \  0 \ &\   0 & \dots & -1
\end{array}\right), \
d_0=\left(\begin{array}{c}
1\\
1\\
\vdots\\
1\\
0\\
0\\
\vdots \\
0
\end{array}\right),$$
where  $m=\sum_{i=1}^l m_i-l,  p=l+m$.  The matrix  $C$  in this
problem is determined by the column vectors
$\overline{c}^{j,k}=\!\!(c^{jk}-c^{jm_j})^T,  j=\!\!1,2,\dots, l;
k=\!\!1,2,\dots,m_j-1,$\\
i.e.
$$C=[\overline{c}^{11},\overline{c}^{12},\ \dots,
\overline{c}^{1r_1-1},\overline{c}^{21},\overline{c}^{22}, \ \dots,
\overline{c}^{2m_2-1}, \dots\dots, \
\overline{c}^{l1},\overline{c}^{l2},\dots,
\overline{c}^{l{m_l-1}}];$$  $g=\sum\limits_{j=1}^lc^{jm_j}$ and \
$e$ \ is the vector determined by the following components \
$$\overline{c}_0^{jk}=c_0^{jk}-c_0^{jm_j},\
 j=1,2,\dots, l;  k=1,2,\dots,m_j-1, $$ i. e.
$$e=(\overline{c}_0^{11},\overline{c}_0^{12},\dots,
\overline{c}_0^{1m_1-1},\overline{c}_0^{21},\overline{c}_0^{22},\dots,
\overline{c}_0^{2m_2-1}, \dots\dots, \
\overline{c}_0^{l1},\overline{c}_0^{l2},\dots,
\overline{c}_0^{l{m_l-1}}).$$ So, by applying Algorithm 2 for the
obtained disjoint bilinear programming problem  with given matrix
$D$ and column vector $d_0$ for system (\ref{eq12}), we obtain the
optimal solution of  the  piecewise linear concave programming
problem (\ref{eq2}), (\ref{eq9})}.

\section{Conclusion}
 A polynomial
algorithm for solving the disjoint bilinear programming problem with
a perfect disjoint subset has been proposed and grounded. This means
that a polynomial algorithm  exists for the disjoint bilinear
programming problem in which one of the disjoint subsets has the
structure of an acute-angled polytope. Taking into account that the
boolean linear programming problem and the piecewise linear concave
programming problem can be represented as disjoint bilinear problems
with an acute-angled polytope for one of the disjoint subsets, these
problems can be solved in polynomial time using Algorithm 2. Based
on this an important conclusion concerned with the coincidence of
classes $P$ and $NP$ (see \cite{fort2,gar,karp,sh1}) can be made.






\end{document}